\documentclass[11pt]{amsart}

\usepackage[all]{xy}
\usepackage{fullpage}
\usepackage{latexsym}
\usepackage{amsmath}
\usepackage{amsfonts}
\usepackage{amssymb}
\usepackage{amsthm}
\usepackage{eucal}
\usepackage{enumerate,yfonts}
\usepackage{mathrsfs}
\usepackage{graphicx}
\usepackage{graphics}
\usepackage{epstopdf}
\usepackage{amscd}
\usepackage{bbm}
\usepackage{hyperref}
\usepackage{url}
\usepackage{color}
\usepackage{bbm}
\usepackage{cancel}
\usepackage{enumerate}
\usepackage{amsmath,amsthm}
\usepackage{amssymb}
\usepackage{epsfig}
\usepackage{pstricks}
\usepackage{xy}
\usepackage{xypic}
\usepackage{bbm}

\newtheorem{thm}{Theorem}[section]
\newtheorem{corollary}[thm]{Corollary}
\newtheorem{lemma}[thm]{Lemma}

\newtheorem{proposition}[thm]{Proposition}
\newtheorem{prop}[thm]{Proposition}

\newtheorem{thm-dfn}[thm]{Theorem-Definition}

\theoremstyle{definition}
\newtheorem{definition}[thm]{Definition}
\newtheorem{remark}[thm]{Remark}
\newtheorem{example}[thm]{Example}

\numberwithin{equation}{section}

\newcommand{\fg}{{\mathfrak g}}
\newcommand{\ft}{{\mathfrak t}}

\newcommand{\fp}{{\mathfrak p}}

\newcommand{\fa}{{\mathfrak a}}

\newcommand{\fc}{{\mathfrak c}}

\newcommand{\bC}{{\mathbb C}}
\newcommand{\bP}{{\mathbb P}}

\newcommand{\bZ}{{\mathbb Z}}

\newcommand{\mE}{\mathcal{E}}

\newcommand{\mO}{\mathcal{O}}
\newcommand{\mL}{\mathcal{L}}

\newcommand{\mG}{\mathcal{G}}

\newcommand{\on}{\operatorname}

\newcommand{\lra}{\longrightarrow}

\newcommand{\ra}{\rightarrow}

\newcommand{\bs}{\backslash}
\newcommand{\is}{\simeq}

\newcommand{\Loc}{\on{LocSys}}
\newcommand{\Bun}{\on{Bun}}

\newcommand{\quash}[1]{}  
\newcommand{\nc}{\newcommand}

\newcommand{\fraka}{{\mathfrak a}}
\newcommand{\frakb}{{\mathfrak b}}

\newcommand{\frakh}{{\mathfrak h}}

\newcommand{\frakp}{{\mathfrak p}}

\newcommand{\frakt}{{\mathfrak t}}
\newcommand{\fraku}{{\mathfrak u}}

\newcommand{\bbA}{{\mathbb A}}

\newcommand{\bbC}{{\mathbb C}}
\newcommand{\bbD}{{\mathbb D}}

\newcommand{\bbG}{{\mathbb G}}
\newcommand{\bbH}{{\mathbb H}}

\newcommand{\bbP}{{\mathbb P}}

\newcommand{\bbR}{{\mathbb R}}
\newcommand{\bbS}{{\mathbb S}}

\newcommand{\bbZ}{{\mathbb Z}}

\newcommand{\calE}{{\mathcal E}}
\newcommand{\calF}{{\mathcal F}}
\newcommand{\calG}{{\mathcal G}}

\newcommand{\calK}{{\mathcal K}}
\newcommand{\calL}{{\mathcal L}}

\newcommand{\calO}{{\mathcal O}}

\newcommand{\calS}{{\mathcal S}}

\newcommand{\calU}{{\mathcal U}}

\newcommand{\calX}{{\mathcal X}}

\newcommand{\calZ}{{\mathcal Z}}

\nc{\al}{{\alpha}} \nc{\be}{{\beta}} \nc{\ga}{{\gamma}}
\nc{\ve}{{\varepsilon}} \nc{\Ga}{{\Gamma}} 
\nc{\La}{{\Lambda}}

\nc{\ad }{{\on{ad }}}

\nc{\aff}{{\on{aff}}} \nc{\Aff}{{\mathbf{Aff}}}

\nc{\der}{{\on{der}}}

\nc{\diag}{{\on{diag}}}

\nc{\Fl}{{\calF\ell}}
\newcommand{\Gal}{{\on{Gal}}}
\newcommand{\Gr}{{\on{Gr}}}
\nc{\Hg}{{\on{Higgs}}}
\newcommand{\Hom}{{\on{Hom}}}

\nc{\Id}{{\on{Id}}}

\nc{\Ind}{{\on{Ind}}}

\nc{\Op}{{\on{Op}}}

\newcommand{\pr}{{\on{pr}}}

\nc{\res}{{\on{res}}}

\newcommand{\Spec}{{\on{Spec}}}

\nc{\tr}{{\on{tr}}}

\newcommand{\GL}{{\on{GL}}}
\nc{\GSp}{{\on{GSp}}} \nc{\GU}{{\on{GU}}} \nc{\SL}{{\on{SL}}}
\nc{\SU}{{\on{SU}}} \nc{\SO}{{\on{SO}}}

\nc{\nh}{{\Loc_{J^p}(\tau')}}
\nc{\bnh}{{\Loc_{\breve J^p}(\tau')}}

\nc{\bU}{{\overline{U}}} \nc{\IC}{{\on{IC}}}

\newcommand{\beqn}{\begin{equation*}}
\newcommand{\eeqn}{\end{equation*}}

\newcommand{\beq}{\begin{equation}}
\newcommand{\eeq}{\end{equation}}

\nc{\QM}{QM}
\nc{\eval}{\textup{ev}}

\newcommand{\gitquot}{/\hspace{-0.2em}/}

\quash{
\topmargin-0.5cm \textheight22cm \oddsidemargin1.2cm \textwidth14cm}
\begin{document}
\title{Real and symmetric quasi-maps}

    \author{Tsao-Hsien Chen} 
        
       \address{School of Mathematics, University of Minnesota, Minneapolis, Vincent Hall, MN, 55455}
       \email{chenth@umn.edu}
       \author{David Nadler} 
        
        \address{Department of Mathematics, UC Berkeley, Evans Hall,
Berkeley, CA 94720}
        \email{nadler@math.berkeley.edu}
       
\maketitle
\begin{abstract}
 Let $G_\bbR$ be a real reductive group 
and let $X$ be the corresponding complex symmetric variety under the Cartan bijection.
We construct a stratified homeomorphism between
the based polynomial arc group of $G_\bbR$ and the based polynomial arc space of $X$.
We also prove a multi-point version
where 
we replace  arcs by
moduli spaces of quasi-maps from the projective line $\mathbb P^1$ to $G_\bbR$ and $X$.
The key ingredients in the proof include:
(i) a multi-point generalization of the ``Gram-Schmidt" factorization of 
loop groups,  and (ii) a nodal degeneration of moduli spaces of quasi-maps.
As an application, we show that  for the closures of real spherical orbits 
in  the real affine Grassmannian, their
singularities near the base point are locally homeomorphic to
complex algebraic varieties.

\end{abstract}
\setcounter{tocdepth}{2} 
\tableofcontents

\newcommand{\bA}{{\mathbb A}}

\nc{\qQM}{\mathscr{QM}}

\nc{\sw}{\textup{sw}}
\nc{\inverse}{\textup{inv}}
\nc{\negative}{\textup{neg}}

\nc{\Sym}{\textup{Sym}}
\nc{\sym}{\mathit{sym}}
\nc{\Orth}{\textup{O}}
\nc{\Maps}{\textup{Maps}}
\nc{\sm}{\textup{sm}}

\nc{\image}{\textup{image}}

\nc{\risom}{\stackrel{\sim}{\to}}


\section{Introduction}

\subsection{Overview}
 Let $G$ be a connected complex reductive Lie group with real form $G_\bbR$, with maximal compact $K_c \subset G_\bbR$, with complexification $K \subset G$,   and corresponding complex symmetric variety $X= K\bs G$. (See Section~\ref{s:group data} for a collection of standard Lie theory constructions downstream of this starting point.)

 This paper gives a direct geometric explanation for why much of the local spherical geometry of the real affine Grassmannian $\Gr_\bbR =  G_\bbR((t))/G_{\bbR}[[t]]$  behaves like complex geometry (for example, satisfying the semisimplicity of the Decomposition Theorem as shown in~\cite{N}). For a spherical orbit $S^\lambda_\bbR =  G_{\bbR}[[t]]\cdot \lambda \subset \Gr_\bbR $, we show the  intersection of its  closure
 $\overline{S^\lambda_\bbR}\subset \Gr_\bbR$ 
 with the open cospherical orbit $T^0_\bbR = G_\bbR[t^{-1}] \cdot 1\subset \Gr_\bbR$ is  
 stratified homeomorphic  to a
complex algebraic variety. More precisely, we view the
 open cospherical orbit as the based polynomial arc group 
 \begin{equation}
 T_\bbR^0 \simeq G_\bbR[t^{-1}]_1 = \{g:\bbP_\bbR^1 \setminus\{0 \} \to G_\bbR \, |\,  g(\infty) = 1\}
 \end{equation}
  and construct a stratified homeomorphism 
  \begin{equation}\label{eq:intro strat homeo}
    G_\bbR[t^{-1}]_1 \simeq   X[t^{-1}]_1 
   \end{equation}
     to the based polynomial arc space of the symmetric variety
  \begin{equation}
  X[t^{-1}]_1 = \{x: \bbP^1 \setminus\{0 \}  \to X \, |\,  x(\infty) = 1\}
\end{equation}

Going further, our main theorem provides a generalization of the stratified homeomorphism \eqref{eq:intro strat homeo} to spaces of maps where we allow poles at multiple points.  
Write $\calG_\bbR \ra\bbR^m$ for the group ind-scheme
whose fiber over 
$(z_1,...,z_m)\in\bbR^m$ is the group 
of maps $\gamma: \bbP_{\bbR}^1\setminus \{z_1,...,z_m\} \to G_\bbR$ such that
$\gamma(\infty) = e$.
Similarly, write  $\calX\ra\bbR^m$ for the  ind-scheme
 whose fiber over 
$(z_1,...,z_m)\in\bbR^m$
is the space of
maps 
$\gamma: \bbP^1\setminus \{z_1,...,z_m\} \to X$ such that 
$\gamma(\infty) = e$.

\begin{thm}[See Theorem~\ref{G=X}]\label{intro G=X}
There is a $K_c$-equivariant stratified homeomorphism 
\begin{equation}\label{eq: intro homeo gr}
\xymatrix{
\calG_\bbR\simeq  \calX
}
\end{equation}
over $\bbR^m$ that restricts to real analytic isomorphisms on spherical strata.
\end{thm}

The construction of the stratified homeomorphism \eqref{eq: intro homeo gr} involves two ingredients of independent interest. 

(1) First, we establish a multi-point generalization of the ``Gram-Schmidt" factorization of the polynomial loop group
\beq\label{eq:intro gs} 
G[t, t^{-1}] \simeq \Omega G_c \cdot G[t]
\eeq
given by multiplication of the factors on the right.
Here $\Omega G_c\subset G[t, t^{-1}]$ is the polynomial based loop group of the maximal compact $G_c \subset G$, i.e., the subgroup of maps that take the unit circle $S^1 \in \bbA^1_\bbC$ to $G_c \subset G$ and  $1\in S^1$ to $1\in G$.

The following states one version of our multi-point generalization; the classical case~\eqref{eq:intro gs} results from taking $S = \emptyset$, $S^+ = \{0\}$, $S^- = \{\infty\}$, and $s_0 = 1$.

\begin{thm}[See  Theorem~\ref{thm: factor}]\label{intro thm: factor}
Let $S^1 \subset \bbA^1_\bbC$ be the unit circle, $\bbD^+\subset \bbA^1_\bbC$ the open unit disk, 
and $\bbD^-\subset \bbP^1_\bbC$ the complementary open disk. Let $\kappa(z) = \bar z^{-1}$ be the conjugation of $\bbP^1_\bbC$ with real form $S^1 \subset \bbA^1_\bbC$.

Suppose given a finite set of points $S\subset S^1$, a finite set of points  $S^+\subset \bbD^+$,
with conjugates $ S^- = \kappa(S^+) \subset \bbD^-$.  Set $\bbS = S \cup S^+ \cup S^-$ to be their union.

Let $G[\bbP^1 \setminus \bbS]$ be the group of polynomial maps $\bbP^1 \setminus \bbS \to G$, and
$G[\bbP^1 \setminus \bbS]_{c}  \subset G[\bbP^1 \setminus \bbS]$ the subgroup that takes $S^1 \setminus S$ to $G_c$.

Fix a point $s_0\in S^1 \setminus S$,  and let $G[\bbP^1 \setminus \bbS]_{c, s_0} \subset G[\bbP^1 \setminus \bbS]_{c} $ be the further subgroup that takes $s_0$ to~$1$.

Then multiplication provides  a homeomorphism
\begin{equation}\label{intro eq: real factor}
\xymatrix{
 G[\bbP^1 \setminus \bbS]_{c, s_0} \times G[\bbP^1 \setminus \{S^-\cup S\}] \ar[r]^-\sim & G[\bbP^1 \setminus \bbS]
}
\end{equation}

\end{thm}

(2) Second, in Section~\ref{s: nodal}, we study how maps, and more generally quasi-maps, behave under collisions of marked points (where the maps or quasi-maps are allowed to have poles) and degenerations of the domain curve itself. 

Specifically, for our application, we study two such families:  the  bubbling of the domain curve $\bbP^1$ to the nodal curve   $\bbP^1\vee \bbP^1$ given in local coordinates by $xy = a^2$, for a parameter $a\in \bbA^1$; and the  collision of  distinct but Galois-conjugate marked points $x \not =  \bar x \in \bbP^1$ of  the fixed domain curve  to a single Galois fixed-point $x_0 = \bar x_0 \in \bbP^1$.

\begin{figure}[h]
\centering
\includegraphics[scale=1.2, trim={4.8cm 21.8cm 3.6cm 4cm},clip]{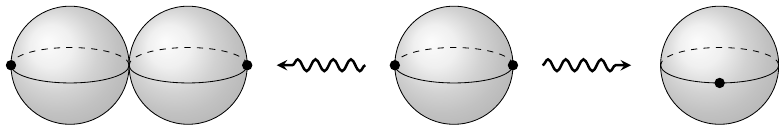}
\caption{Degenerations of curve $\bbP^1$ and marked points $S$. Nodal degeneration of curve $\bbP^1$ to the left, collision of marked points $S$ to the right.}
\end{figure}

A detailed study of quasi-maps under the collision of distinct but Galois-conjugate marked points
can be found in \cite{CN1}. Here let us focus on the key Galois-theoretic property of the bubbling degeneration that we exploit. We start at the parameter $a = 1$ with the conjugation of  $\bbP^1$ given by $x\mapsto \bar x^{-1}$
in the local coordinate $x$. So its real form is the unit circle $S^1 \subset \bbP^1_{\bbC}$ and it exchanges the points $0$ and $\infty$. Then we  extend this over $a \not  = 0 $  by taking $x\mapsto \bar a^2 \bar x^{-1}$.  
When we degenerate to $a = 0$, we find the conjugation $(x, y) \mapsto (\bar y, \bar x)$ exchanging the two components of $\bbP^1\vee \bbP^1$. Thus  the Galois-action on the special fiber is almost free with the node the only real point. This allows us to describe   Galois-equivariant maps, and more generally quasi-maps,   
on $\bbP^1\vee \bbP^1$ in almost completely complex algebraic terms: the real structure constrains their values at the node alone. 

Finally, we identify spaces of Galois-equivariant maps, and more generally quasi-maps, that are stratified locally constant over real   parameters $a\in \bbA_\bbR^1$. This provides homeomorphisms from their general fibers, which involve evident real structures, to their special fibers, where the real structures ``disappear" due to the almost free nature of the Galois-action. 
The proof of Theorem~\ref{intro G=X} is an application of this idea:  the general fiber is the group 
$\calG_\bbR$ of real maps and the special fiber $\calX$ is also a space  of ``real" maps but has a completely complex algebraic interpretation.

\subsection{Acknowledgements} 
The authors would like to thank the generous organizers of the  2019 SE Lie Theory workshop at LSU 
where these results were presented.

T.-H. Chen would also like to 
thank the Institute of Mathematics Academia Sinica in Taipei for support, hospitality, and a nice research environment.

The research of
T.-H. Chen is supported by NSF grant  DMS-2001257
and DMS-2143722, and that of D. Nadler  by NSF grant  DMS-2101466.

\section{Group data}\label{s:group data}
 We collect here notation and standard constructions used throughout the rest of the paper
(for further discussion, see for example~\cite{CN1,N}).
 
 \subsection{Real forms}
 
 Let $G$ be a connected complex reductive Lie group with Lie algebra $\fg$.

Let $G_\bbR\subset G$ be a real form, defined by a conjugation $\eta:G\to G$,
with Lie algebra $\fg_\bbR \subset \fg$.

Choose a Cartan conjugation $\delta:G\to G$ that commutes with $\eta$, and let $G_c\subset G$ be the corresponding maximal compact subgroup with Lie algebra $\fg_c \subset \fg$.

 Introduce the involution $\theta = \delta \circ \eta:G\to G$, and let $K \subset G$ be the fixed subgroup of $\theta$.

One can organize the above groups into the diagram:
\beq\label{eq:diamond}
\xymatrix{&G&\\
K\ar[ru]\ar[ru]&G_c \ar[u]&G_\bbR\ar[lu]\\
&K_c\ar[lu]\ar[u]\ar[ru]&}
\eeq
Here 
$K_c$ is the fixed subgroup of $\theta, \delta$, and $\eta$ together (or any two of the three) and the maximal compact subgroup of $G_\mathbb R$ with complexification $K$.

Fix a maximal $\delta$-stable split torus $A_\bbR\subset G_\bbR$ with Lie algebra $\fraka_\bbR\subset \fg_\bbR$,
and complexification $A\subset G$ with Lie algebra $\fraka\subset \fg$. Fix a maximal  $\delta$-stable torus $T_\bbR\subset G_\bbR$ containing $A_\bbR$  with Lie algebra $\frakt_\bbR\subset \fg_\bbR$,
and complexification $T\subset G$ with Lie algebra $\frakt\subset \fg$. 

 Fix a Borel subgroup  $B\subset G$  containing $T$ with Lie algebra $\frakb\subset \fg$, and unipotent radical $U\subset B$ with Lie algebra $\fraku \subset \frakb$. Let $H = B/U$ be the universal Cartan with Lie algebra 
 $\frakh = \frakb/\fraku$. 
 Note 
the composition 
$T \to B \to H$
is an isomorphism.

Let $W=N_G(\frakt)/Z_G(\frakt)= N_G(T)/T$ denote the Weyl group,
and $W_0 =N_K(\fa)/Z_K(\fa)$ the ``baby Weyl group".

Let $\Lambda_T = \Hom(\bbG_m, T)$ denote the coweight lattice of $T$ with dominant coweights $\Lambda^+_T\subset \Lambda_T$. Let $R_G \subset \Lambda_T$ denote the coroot lattice generated by the simple coroots $\Delta_G\subset R_G$
and let $R_G^+=R_G\cap\Lambda_T^+$.
 Let $\pi_1(G)$ denote the fundamental group based at the identity $e\in G$, and recall the natural isomorphism $\Lambda_T/R_G \risom \pi_1(G)$.

Similarly, let $\Lambda_A = \Hom(\bbG_m, A)$ denote the coweight lattice of $A$ with dominant coweights $\Lambda^+_A =  \Lambda_A\cap \Lambda_T^+$.

Introduce the  symmetric variety  $X= K\bs G$.
The map $\pi:G\to G$, $\pi(g)=\theta(g)^{-1}g$
factors through $X$, and descends to an isomorphism $X\risom G_\sym^0$, where $G_\sym = \{g\in G \, |\, \theta(g^{-1}) = g\}$, and $G_\sym^0\subset G_\sym$ denotes its neutral component.
Note the inclusion $A \subset G_\sym^0$ induces a map $\Lambda_A \to \pi_1(G_\sym) \simeq \pi_1(X)$.

Consider the induced map $\pi_*:\pi_1(G)\to \pi_1(X)$. Define $\calL\subset\Lambda_A$
to be the inverse image of $\pi_*(\pi_1(G))\subset \pi_1(X)$ under the natural map 
$\Lambda_A\to\pi_1(X)$, and
set $\calL^+ = \calL \cap \Lambda^+_T$.
Note that $\mL=\Lambda_A$ if and only if $K$ is connected.

\begin{example}[Complex groups]
An important special case is when the real form is itself a complex group. To avoid potential confusion in this case, we will use the alternative notation $\sw:G\times G\to G\times G$, $\sw(g, h) = (h, g)$ for the ``swap" Cartan involution with fixed-point subgroup
the diagonal $G\subset G\times G$. 
The corresponding conjugation  is the composition  $\sw_\delta =\sw\circ (\delta\times \delta)$, and
the compact conjugation is the product $\delta \times \delta$.
We identify the corresponding symmetric variety with the group $G \simeq G\bs(G\times G)$ via inclusion to the left factor, which coincides with 
 the subspace  $G\simeq \{(g, h)\in G\times G \, |\, h = g^{-1}\}$.
  \end{example}

  Introduce  the affine quotient
\begin{equation}
\xymatrix{
\fc = \ft\gitquot{W} = \Spec( \calO(\ft)^{W}) \simeq \fg/\hspace{-0.25em}/G =  \Spec(\calO(\fg)^G)
}
\end{equation}

The conjugation $\eta$ descends to a real structure on $\fc$, and we denote its real points by $\fc_\bbR$.
We also have the real characteristic  polynomial map  $\fg_\bbR \to \fc_\bbR$ from real matrices to their unordered eigenvalues.


 \subsection{Loop groups}

Let $\calK = \bbC((z))$ denote the field of Laurent series, $\calO = \bbC[[z]]$ the ring of power series, and $\calO^- = \bbC[z^{-1}]$ the ring of Laurent poles.
Let $\Gr = G(\calK)/G(\calO)$ be the affine Grassmannian of $G$.\footnote{Our concerns in this paper will be exclusively topological,
and we will ignore any non-reduced structure throughout.} 
For any $g\in G(\calK)$ we denote by $[g]\in\Gr$ the 
corresponding coset.

Via the natural inclusion $\Lambda_T\to T(\calK) \subset G(\calK)$, any coweight $\lambda \in \Lambda_T$ defines a point  $[\lambda] \in \Gr$.
For a dominant coweight $\lambda\in \Lambda_T^+$, introduce the $G(\calO)$-orbit $S^\lambda 
= G(\calO) \cdot [\lambda]\subset \Gr $ (spherical stratum)
and $G(\calO^-)$-orbit $T^\lambda 
= G(\calO^-) \cdot [\lambda] \subset \Gr$  (cospherical stratum).
Recall the disjoint union decompositions
\begin{equation}\label{eq:strat}
\xymatrix{
\Gr = \coprod_{\lambda\in \Lambda_T^+}   S^\lambda & \Gr= \coprod_{\lambda\in \Lambda_T^+} T^\lambda
}
\end{equation}

We can similarly repeat the above constructions over the real numbers.
Let $\calK_\bbR = \bbR((z))$ denote the field of real Laurent series, 
 $\calO_\bbR = \bbR[[z]]$ the ring of real power series,
 and  $\calO^- = \bbR[z^{-1}]$ the ring of real Laurent poles.
  Let $\Gr_\bbR=G_\bbR(\calK_\bbR)/G_\bbR(\mO_\bbR)$ be the real affine Grassmannian of the real form $G_\bbR$.

For a dominant coweight $\lambda\in \Lambda_A^+  = \Lambda_A \cap \Lambda_T^+$, 
set $S^\lambda_{\bbR} 
=  G_\bbR(\calO_\bbR) \cdot [\lambda] =  S^\lambda\cap \Gr_\bbR $ (real spherical stratum)
and $T_{\bbR}^{\lambda} =  G_\bbR(\calO^-_\bbR) \cdot [\lambda]
= T^\lambda\cap \Gr_\bbR $ 
 (real cospherical stratum).
We have the disjoint union decompositions
\begin{equation}\label{eq:strat}
\xymatrix{
\Gr_\bbR = \coprod_{\lambda\in \Lambda_A^+}   S^\lambda_\bbR & \Gr_\bbR= \coprod_{\lambda\in \Lambda_A^+} T_\bbR^\lambda
}
\end{equation}

Note that all of above constructions result  from the natural  conjugation
$\eta_\calK:G(\calK) \to G(\calK)$, $\eta_{\calK}(g(z)) = \eta(g(\bar z))$ with real form 
$G_\bbR(\calK_\bbR) \subset G(\calK)$.

Now let us recall some parallel constructions where  we work with the global curve $\bbG_m = \Spec(\bbC[z, z^{-1}]) = \bbP^1 \setminus \{0,\infty\}$ in place of the punctured disk $D^\times = \Spec (\calK)$,
and
 similarly  $\bbA^1 = \Spec(\bbC[z]) = \bbP^1 \setminus \{\infty\}$ 
  in place of the disk $D = \Spec(\calO)$.

Introduce  the polynomial loop group $LG = G(\bbC[z, z^{-1}]) \subset G(\calK)$ of  maps $\bbG_m\to G$,  and similarly, the polynomial arc group $L_+ G = G(\bbC[z]) \subset G(\calO)$ of maps $\bbA^1\to G$. 
 Recall the natural map  is an isomorphism $LG/L_+G \risom  G(\calK)/G(\calO) =   \Gr$.

We have an additional compact conjugation
$\kappa:\bbG_m \to \bbG_m$,
$\kappa(z) = \bar z^{-1}$ with real points the unit circle $S^1\subset \bbC^\times$.\footnote{There is another  conjugation $z\mapsto -\bar z^{-1}$ of the curve $\bbG_m$ with empty real points, but  it will not appear in the developments of this paper.}    Note that $\kappa$ does not preserve the punctured disk $D^\times \subset \bbG_m$.

We extend the conjugations $\eta, \delta:G\to G$ to conjugations $ \eta, \delta:LG\to LG$ by the formulas $ \eta(g)(z) = \eta(g(\kappa(z))), \delta(g)(z) = \delta(g(\kappa(z)))$. The corresponding real forms 
 $L G_\bbR, LG_c\subset LG$ consist of maps $g:\bbG_m\to G$ that take $S^1\subset \bbG_m$ respectively to $G_\bbR, G_c \subset G$. We denote by $\Omega G_\bbR\subset LG_\bbR$, $ \Omega G_c \subset LG_c$ the based subgroups of maps that take $1\in S^1$ to the identity $e\in G$. Note that multiplication gives isomorphisms
$\Omega G_\bbR \times G_\bbR \risom LG_\bbR$, $\Omega G_c \times G_c \risom LG_c$.

%


\section{Factorization}

In this section, we record a useful extension of a well-known loop group factorization.

Recall the ``Gram-Schmidt"
 factorization
  \beq \label{eq:factor}
  \xymatrix{
  \Omega G_c \times L_+ G  \ar[r]^-\sim &  LG
  }
  \eeq
of the polynomial loop group $LG = G(\bbC[z, z^{-1}])$, its arc subgroup $L_+ G = G(\bbC[z])$, and the based loop group $\Omega G_c \subset LG$. Note that \eqref{eq:factor} is equivalent to the fact that the $  \Omega G_c$-action
on the base point $[e]\in  \Gr \simeq LG/L_+ G$ induces a homeomorphism 
  \begin{equation}\label{eq:loop=gr}
  \xymatrix{
  \Omega G_c\ar[r]^-\sim & \Gr 
  }
  \end{equation}
%

Consider the projective line $ \bbP^1 = \on{Proj}(\bbC[z_0, z_1])$. Using the local coordinate $z = z_1/z_0$,  we will regard the complex points of $\bbP^1$ as the Riemann sphere $\bbP^1(\bbC) = \bbC \cup \{\infty\}$.

Given any finite set of points $\bbS \subset\bbP^1$, we will denote by $G[\bbP^1 \setminus \bbS]$
the group ind-scheme of maps  $\bbP^1\setminus \bbS\to G$. We will only be interested in the complex points of 
 $G[\bbP^1 \setminus \bbS]$ equipped with the classical topology, and thus ignore any non-reduced structure.
 
 For example, if $\bbS=\emptyset$, we have $G[\bbP^1] = G$;
if $\bbS=\{\infty\}$, we have $G[\bbP^1\setminus \{\infty\}] = L_+G$; and
if $\bbS=\{0, \infty\}$,  we have $G[\bbP^1\setminus \{0, \infty\}] = LG.$

For a pair $s= (s_+, s_-)$ of points $s_+, s_- \in \bbP^1$, we will write $L_s G = G[\bbP^1 \setminus \{s_+, s_-\} ]$. If we choose an isomorphism $\varphi:\bbP^1\risom \bbP^1$, $\varphi(0) = s_+$, $\varphi(\infty) = s_-$, then we obtain an isomorphism
\beq\label{eq:pullback isom}
\xymatrix{
\varphi^*:L_s G \ar[r]^-\sim & LG & \varphi^*(g) = g \circ \varphi
}
\eeq
 To uniquely prescribe $\varphi$, we can further require $\varphi(p) = q$, for some $p, q\in \bbP^1$
 with $p \not = 0, \infty$, $q\not = s_+, s_-$.

\begin{lemma}\label{l:factor}
Suppose $S^+, S^-\subset \bbP^1$ are disjoint, finite sets of points with $S^-$ non-empty.

For $\bbS = S^+ \cup S^-$, the natural map is an isomorphism
\beq
\xymatrix{
G[\bbP^1 \setminus \bbS]/G[\bbP^1 \setminus S^-] \ar[r]^-\sim & \prod_{s\in S^+} \Gr_{s}
}
\eeq
\end{lemma}

\begin{proof}
Set $U^\pm = \bbP^1 \setminus S^\pm$. Since $ S^-$ is non-empty,  the restriction $\calE|_{U^-}$  of any $G$-bundle $\calE$ on $\bbP^1$ may be trivialized. Thus the left hand side classifies a trivial $G$-bundle   $\calE_0^+$ on $U^+$, a $G$-bundle $\calE^-$ on $U^-$, and an isomorphism between them over $U^+\cap U^-$. This is equivalent to a $G$-bundle $\calE$ on $\bbP^1$ with a trivialization over $U^+$. It is standard that this factorizes to give the right hand side.
\end{proof}

Fix
a non-zero real number $a$  and
consider the conjugation  $\kappa_a: \bbP^1\to \bbP^1$, 
$\kappa(z) = \frac{1}{a^2\bar z}$ with real points the  circle $S^1_{a^{-1}} = \{ |z| = |a^{-1}| \}$.
Let $\bbD^+_a = \{|z| < |a^{-1}|\}$ denote the open  disk,
and $\bbD^-_a = \{|z| > |a^{-1}|\}$  the complementary open disk.

If  a finite set of points $\bbS \subset\bbP^1$ is invariant under $\kappa_a$, then the conjugation $\delta:G\to G$ induces a conjugation $\delta: G[\bbP^1 \setminus \bbS]\to G[\bbP^1 \setminus \bbS]$, $\delta(g)(z) = \delta(g(\kappa_a(z))$.
The corresponding real points, denoted by $G[\bbP^1 \setminus \bbS]_c$, consist of maps $g:\bbP^1 \setminus \bbS\to G$ 
 that take  $S^1 \setminus (S^1\cap \bbS)$ to the maximal compact $G_c\subset G$.
 For example, if $\bbS=\emptyset$, we have $G[\bbP^1]_c = G_c$;
if $\bbS=\{\infty\}$, we have $G[\bbP^1\setminus \{\infty\}]_c = G_c$.

For a pair $s= (s_+, s_-)$ of points $s_+\in \bbD^+_a$, $s_- =\kappa_a(s_+) \in \bbD_a^-$, we write $L_s G_c = G[\bbP^1 \setminus \{s_+, s_-\} ]_c$. 
Observe that the map $\varphi:\bbP^1\risom \bbP^1$, $\varphi(0) = s_+$, $\varphi(\infty) = s_-$,
$\varphi(a^{-1}) = a^{-1}$ commutes with the conjugation $\kappa_a$. Thus the isomorphism~\eqref{eq:pullback isom}
restricts to an isomorphism on real points
\beq\label{eq:pullback isom real}
\xymatrix{
\varphi^*:L_s G_c \ar[r]^-\sim & LG_c 
}
\eeq
For any $s_0\in S^1$, we will write $\Omega_{s, s_0} G_c \subset L_s G_c$ for the based subgroup of those maps with $g(s_0) = e$.

For finite subsets $\bbS \subset S^1$, we have the following:

\begin{lemma}\label{l:const}
For any finite subset  $\bbS \subset S^1$, any element of  $G[\bbP^1\setminus \bbS]_c$ is constant, and hence 
$G[\bbP^1\setminus \bbS]_c = G_c$.
\end{lemma} 

\begin{proof}
Choose an embedding $G\subset \GL(N)$ so that $G_c\subset U(N)$. For any $g\in G[\bbP^1\setminus \bbS]_c$, consider a matrix entry $g_{ij}: \bbP^1\setminus \bbS\to \bbA^1$. 
The   restriction $g|_{S^1 \setminus \bbS}$ lands in $G_c \subset U(N)$, hence the restricted matrix entry $g_{ij}|_{S^1 \setminus \bbS}$ is bounded. Thus $g_{ij}$ is bounded, and hence constant. 
\end{proof}

Now we are ready to state the main result of this section. 
The special case when $S=\emptyset$ and  $S^+ = \{0\}$, so that $S^- = \{\infty\}$,  recovers the
``Gram-Schmidt"
 factorization~\eqref{eq:factor}.

\begin{thm}\label{thm: factor}

Suppose given a finite set of points $S\subset S^1$, a finite set of points  $S^+\subset \bbD^+$,
with conjugates $ S^- = \kappa(S^+) \subset \bbD^-$.  Set $\bbS = S \cup S^+ \cup S^-$ to be their union.

Fix a point $s_0\in S^1 \setminus S$, and consider the kernel of the evaluation
 \beq
  \xymatrix{
  G[\bbP^1 \setminus \bbS]_{c, s_0} := \ker(\on{ev}_{s_0}:G[\bbP^1 \setminus \bbS]_{c} \ar[r] &  G_c)
  }
\eeq

Then we have:

1) Multiplication provides  a homeomorphism
\begin{equation}\label{eq: real factor}
\xymatrix{
 G[\bbP^1 \setminus \bbS]_{c, s_0} \times G[\bbP^1 \setminus \{S^-\cup S\}] \ar[r]^-\sim & G[\bbP^1 \setminus \bbS]
}
\end{equation}

2) The natural action map provides a homeomorphism
\begin{equation}\label{eq: real homeo}
\xymatrix{
 G[\bbP^1 \setminus \bbS]_{c,s_0} \ar[r]^-\sim & G[\bbP^1 \setminus \bbS]/G[\bbP^1 \setminus \{S^-\cup S\}]  \simeq   
 \prod_{s\in S^+} \Gr_{s}
}
\end{equation}

3)
For any ordering $S^+ = \{ s_1, \ldots, s_k\}$, multiplication provides  a homeomorphism
\begin{equation}\label{eq: real ordered factor}
\xymatrix{
 \Omega_{s_1, s_0} G_c \times \cdots \times \Omega_{s_k, s_0} G_c  \ar[r]^-\sim & 
 G[\bbP^1 \setminus \bbS]_{c, s_0}
}
\end{equation}

\end{thm}

%
%
%
%
%
%

\begin{proof}
Clearly 1) and 2) are equivalent. We will prove 2) and 3) simultaneously by
 considering the diagram
\begin{equation*}
\xymatrix{
  \Omega_{s_1, s_0} G_c \times \cdots \times \Omega_{s_k, s_0} G_c\ar[r] & G[\bbP^1 \setminus \bbS]_{c,s_0} \ar[r] & G[\bbP^1 \setminus \bbS]/G[\bbP^1 \setminus \{S^-\cup S\}]  \simeq   
 \prod_{s\in S^+} \Gr_{s}
}
\end{equation*}

Clearly the composite map $ \Omega_{s_1, s_0} G_c \times \cdots \times \Omega_{s_k, s_0} G_c \to  \prod_{s\in S^+} \Gr_{s}$ is a homeomorphism by~\eqref{eq:loop=gr}. Thus it suffices to show
the map $G[\bbP^1 \setminus \bbS]_{c,s_0}\to \prod_{s\in S^+} \Gr_{s}$ is injective, i.e.,
the $G[\bbP^1 \setminus \bbS]_{c,s_0}$-action on the base-point $[e]\in \prod_{s\in S^+} \Gr_{s}$ is free. Any element in the stabilizer must extend across $S^+$, hence also across $S^-$, and so lie in $G[\bbP^1 \setminus S]_{c, s_0}$.
Thus by Lemma~\ref{l:const}, any element in the stabilizer must be constant. Since its evaluation at $s_0$ is trivial, the element must  in fact be trivial.
\end{proof}


\section{Quasi-maps}\label{quasi-maps}
\subsection{Definitions}

Let $B$ be a smooth complex base-scheme.

Let $\pi:\calZ\to B$ be a projective family of curves, with fibers denoted $\calZ_b = \pi^{-1}(b)$. 
We do not assume the total space $\calZ$ or  fibers $\calZ_b$  are smooth. See the next section for the specific setting used in this paper.

Let $\Bun_G(\calZ/B)$ denote the moduli stack of a point $b\in B$ and a $G$-bundle $\calE$ on the fiber $\calZ_b$.
More precisely, an $S$-point consists of an $S$-point $a:\Spec S\to B$ and a $G$-bundle $\calE$ on the 
corresponding fiber product $\calZ \times_{B} \Spec S$. 
Denote by $p:\Bun_G(\calZ/B)\to B$  the evident projection with fibers $p^{-1}(b) = \Bun_G(\calZ_b)$.

Let $\sigma = (\sigma_1,\ldots, \sigma_n):B\to\calZ^n$ be an ordered $n$-tuple of sections of $\pi$.
We allow the  sections to intersect or coincide, but require $\sigma(b)\in \calZ^n_b$ to be a smooth point, for all $b\in B$.

For an affine $G$-variety $X$, 
let $\QM_{G, X}(\calZ/B, \sigma)$ denote the ind-stack of quasi-maps   classifying
a point $b\in B$, a $G$-bundle $\calE$ on the fiber $\calZ_b$, and a section
\begin{equation}
\xymatrix{
s:\calZ_b\setminus \{\sigma_1(b), \ldots, \sigma_n(b)\} \ar[r] & X_\calE
}
\end{equation}
to the associated $X$-bundle over the complement of the points $\sigma_1(b),\ldots,  \sigma_n(b)\in \calZ_b$.
We have the evident forgetful maps
\beq
\xymatrix{
q:\QM_{G, X}(\calZ/B, \sigma)\ar[r] & \Bun_G(\calZ/B)\ar[r] &  B
}
\eeq
with fibers  $q^{-1}(b) = \QM_{G, X}(\calZ_b, \sigma(b))$.

Let $\xi:B\to\calZ$ be another section of $\pi$ such that 
$\xi(b)\neq\sigma_i(b)$, for all $b\in B$,  and $i=1,\ldots, n$. 

Let $\QM_{G, X}(\calZ/B, \sigma,\xi)$ denote the ind-stack of rigidified quasimaps classifying quadruples $(b,\mE,s,\iota)$
where $(b,\mE,s)$ is a quasi-map as above and $\iota:\mE_K|_{\xi(b)}\is K$ is a trivialization, where 
$\mE_K$ is the $K$-reduction of $\mE$ on $\calZ_b\setminus \{\sigma_1(b), \ldots, \sigma_n(b)\}$ given by the section $s$.
We have the evident forgetful maps
\beq
\xymatrix{
r:\QM_{G, X}(\calZ/B, \sigma,\xi)\ar[r] & \Bun_G(\calZ/B)\ar[r] &  B
}
\eeq
with fibers 
$r^{-1}(b) = \QM_{G, X}(\calZ_b, \sigma(b),\xi(b))$.

Suppose given conjugations $c_\calZ:\calZ\to\calZ$ and $c_{B}:B\to B$ such that $\pi\circ c_\calZ = c_B \circ \pi$.

The conjugation $\eta$ of $G$ induces a conjugation $\eta:\Bun_G(\calZ/B)\to \Bun_G(\calZ/B)$ given by
$
\eta(b, \calE) = (c_{B}(b), c_{\calZ}^*\calE_\eta)
$
 where 
we write $c_{\calZ}^*\calE_\eta
$ for the bundle $c_{\calZ}^*\mE$ with  its $\eta$-twisted $G$-action.
We denote by $\Bun_G(\calZ/B)_\bbR$ the corresponding real points.

Suppose $n=2m$ so that we have
$\sigma = ( \sigma^+,  \sigma^-):B\to \calZ^n$ with components $ \sigma^\pm = (\sigma^\pm_1, \ldots, \sigma^\pm_m):B\to \calZ^m$.
Suppose further that $c_\calZ\circ \sigma^\pm_i \circ c_{B}=\sigma^\mp_i$, for $i=1, \ldots, m$.
Then the conjugation $\eta$ of $G$  induces a conjugation $\eta:\QM_{G, X}(\calZ/B, \sigma)\to \QM_{G, X}(\calZ/B, \sigma)$. We denote by $\QM_{G, X}(\calZ/B, \sigma)_\bbR$ the  corresponding real points.

Suppose further that  $c_\calZ\circ \xi = \xi \circ c_B$. 
Then the conjugation $\eta$ of $G$ similarly induces a  conjugation $\eta: \QM_{G, X}(\calZ/B, \sigma, \xi)\to \QM_{G, X}(\calZ/B, \sigma, \xi)$. We denote by $\QM_{G, X}(\calZ/B, \sigma, \xi)_\bbR$
the  corresponding real points.

%
%
%

\subsection{Uniformizations}\label{uniformization}
Let $\Gr_{G,\calZ,\sigma_i}$ (resp. $\Gr_{G,\calZ,\sigma}$)
denote the Beilinson-Drinfeld Grassmannian 
of a point $b\in B$, a $G$-bundle $\mE$ on $\calZ_b$, and a section 
$s:\calZ_b\setminus\sigma_i(b)\to\mE$ (resp. $s:\calZ_b\setminus\{\sigma_1(b),\ldots, \sigma_n(b)\}\to\mE$). 

Let $G[\calZ,\hat\sigma_i]$ denote the group scheme of a point $b\in B$
and a section $\hat D_{\sigma_i(b)}\to G$, where $\hat D_{\sigma_i(b)}$
is the formal disk around $\sigma_i(b)$.

Let $G[\calZ,\sigma_i]$ (resp. $G[\calZ,\sigma]$) denote the  group ind-scheme of a point $b\in B$
and a section $s:\calZ_b\setminus\{\sigma_i(b)\}\to G$ (resp. $s:\calZ_b\setminus\{\sigma_1(b),\ldots, \sigma_n(b)\}\to G$). 

Let $G[\calZ,\sigma_i,\xi]$ (resp. $G[\calZ,\sigma,\xi]$)
denote the subgroup ind-scheme of $G[\calZ,\sigma_i]$ (resp. $G[\calZ,\sigma]$)
consisting of $(b,s)\in G[\calZ,\sigma_i]$ (resp. $(b,s)\in G[\calZ,\sigma]$)
such that $s(\xi(b))=e$. 

For any $b\in B$, we write 
$\Gr_{G,\calZ,\sigma,b}$, $G[\calZ,\sigma, b]$, etc., for the respective fibers over $b$.

The conjugations
$c_\calZ$, $c_{B}$, $\eta$ induce conjugations on 
$\Gr_{G,\calZ,\sigma}$, $G[\calZ,\sigma]$, etc., and we denote by 
$\Gr_{G,\calZ,\sigma,\bbR}$, $G[\calZ,\sigma]_\bbR$, etc., the respective 
real points.

For any $b\in B(\bbR)$, we write 
$\Gr_{G,\calZ,\sigma,b,\bbR}$, $G[\calZ,\sigma, b]_\bbR$, etc., for the respective fibers over $b$.

The group ind-scheme $G[\calZ,\sigma]$ (resp.
$G[\calZ,\sigma_i]$) naturally acts on 
$\Gr_{G,\calZ,\sigma}$ (resp. $\Gr_{G,\calZ,\sigma_i}$) and we have uniformizations morphisms
\beq
\xymatrix{
G[\calZ,\sigma_i]\backslash\Gr_{G,\calZ,\sigma_i}\ar[r] & \Bun_G(\calZ/B)
&
G[\calZ,\sigma]\backslash\Gr_{G,\calZ,\sigma}\to\Bun_G(\calZ/B)
}\eeq
\beq\label{eq:qm almost unif}
\xymatrix{
K[\calZ,\sigma]\backslash\Gr_{G,\calZ,\sigma}\ar[r] &\QM_{G,X}(\calZ/B,\sigma)
}
\eeq
\beq\label{eq:qm unif}
\xymatrix{
K[\calZ,\sigma,\xi]\backslash\Gr_{G,\calZ,\sigma}\ar[r] & \QM_{G,X}(\calZ/B,\sigma,\xi)
}
\eeq

The uniformizations are compatible with the given conjugations, hence 
induce uniformizations on real points:
\beq
\xymatrix{
G[\calZ,\sigma_i]_\bbR\backslash\Gr_{G,\calZ,\sigma_i,\bbR}\ar[r] &\Bun_G(\calZ/B)_\bbR
&
G[\calZ,\sigma]_\bbR\backslash\Gr_{G,\calZ,\sigma,\bbR}\ar[r] &\Bun_G(\calZ/B)_\bbR
}
\eeq
\beq\label{eq:qm almost unif real}
\xymatrix{
K[\calZ,\sigma]_\bbR\backslash\Gr_{G,\calZ,\sigma,\bbR}\ar[r] &\QM_{G,X}(
\calZ/B,\sigma)_\bbR
}
\eeq
\beq\label{eq:qm unif real}
\xymatrix{
K[\calZ,\sigma,\xi]_\bbR\backslash\Gr_{G,\calZ,\sigma,\bbR}\ar[r] &\QM_{G,X}(\calZ/B,\sigma,\xi)_\bbR
}
\eeq

\subsection{Morphisms}\label{morphisms}
Let $G_1$ and $G_2$ be two reductive groups with 
complex conjugations $\eta_1$ and $\eta_2$ and Cartan involutions $\theta_1$ 
and $\theta_2$ respectively. 
Then the constructions of quasi-maps, rigidified quasi-maps, uniformization
morphisms, and real forms of those are functorial with respect to 
homomorphisms
$f:G_1\to G_2$ that intertwine
$\eta_1,\eta_2$ and $\theta_1,\theta_2$.


\section{Nodal degeneration}\label{s: nodal}

We invoke here the preceding constructions  in a situation to be studied in the remainder of the paper.

\subsection{Universal family}\label{nodal curve}

Let $\bA^1 = \Spec(\bbC[a])$ be the affine line with coordinate $a$. 
Consider the product $\bP_x^1 \times \bP_y^1$  with respective homogeneous coordinates $[x_0, x_1], [y_0, y_1]$,
local coordinates $x = x_1/x_0, y= y_1/y_0$,
and projections $p_x, p_y:\bP_x^1 \times \bP_y^1\to \bP^1$. For convenience, 
we will also set $t^+ = x^{-1}, t^- = y^{-1}$.

Introduce the surface $Z \subset \bP_x^1 \times \bP_y^1 \times \bA^1$ cut out by $x_1 y_1 = a^2 x_0 y_0$.
We will regard $Z$ as a family of curves via the evident projection $p:Z \to \bA^1$. We denote the fibers by $Z_a = p^{-1}(a)$,
for $a\in \bA^1$. When $a \not = 0$, projection along $p_x$ or $p_y$ provides an isomorphism $Z_a \simeq \bP^1$.
When $a = 0$, the image of the inclusion $Z_0\subset \bP^1 \times \bP^1$ is the nodal curve 
\begin{equation}
\xymatrix{
\bP^1_x\vee \bP^1_y := (\bP_x^1 \times \{0\}) \cup (\{0\} \times  \bP_y^1)
}
\end{equation}

Equip $\bbA^1$ with the usual conjugation $c(a) = \bar a$ with real points  $\bA^1(\bbR) \simeq \bbR$.
Equip $Z$ with the twisted conjugation
$c_Z(x, y, a) = (\bar y,  \bar x, \bar a)$. 
When $a\not = 0 \in \bA^1(\bbR)$, under the identification $p_x:Z \risom \bP^1_x$, we have $c_Z(x) = a^2/\bar x $, and
thus $p_x: Z_a(\bbR)\risom \bbR\bbP^1$. 
When $a= 0 \in \bA^1(\bbR)$, the components of $Z_0$  are exchanged by $c_Z$, and $Z_0(\bbR)$ is the single  point $x=0, y=0$.

We will also use the 
coordinates $t^+=x^{-1}, t^-=y^{-1}$.
Note that in terms of the coordinate $t^+$ of $\mathbb P^1$
the complex conjugation is given by 
$c_\calZ(t^+)=\frac{1}{a^2\bar t^+}$.
Let $\bA_{+}^1 = \Spec(\bbC[t^+])$, $\bA_{-}^1 = \Spec(\bbC[t^-])$ be the affine lines with 
respective coordinates $t^+=x^{-1}, t^-=y^{-1}$. Thus we have natural open embeddings 
$\bA_{+}^1 \subset \bP^1_x$, $\bA^1_{-} \subset \bbP^1_y$.

Fix $n=2m$. 
Consider the base scheme 
\beq
\xymatrix{
B = \bbA_a^1 \times B^+ \times B^- 
&
B^\pm = (\bA_{\pm}^1)^m 
}
\eeq
with 
coordinates $(a, t^+, t^-)$ where $t^\pm = (t^\pm_1, \ldots, t^\pm_m)$. 
Equip $B$ with the twisted conjugation $c_A(a, t^+, t^-) = (\bar a, \bar t^-, \bar t^+)$
 where $\bar t^\pm = (\bar t^\pm_1, \ldots, \bar t^\pm_m)$. 
 Projection provides an identification of real points
\beq\label{B(R)}
\xymatrix{
B(\bbR) \simeq \bbA^1(\bbR) \times B^+(\bbC) \simeq \bbR\times \bbC^m
}
\eeq

Consider the   family of  curves 
\beq
\xymatrix{
p:\calZ=Z\times_{\bbA^1} B \ar[r] &  B
}
\eeq
with the tautological sections
\beq
\xymatrix{
\sigma=(\sigma^+, \sigma^-):B \ar[r] & \calZ^{2m}
}
\eeq
\beq
\xymatrix{
\sigma_i^+(a, t^+,t^-) = (a,  [t_i^+, 1], [1, t_i^+ a^2], t^+,t^-) 
& i = 1, \ldots, m
}
\eeq
\beq
\xymatrix{
\sigma_i^-(a, t^+, t^-) = (a, [1, t_i^- a^2], [t_i^-, 1],t^+, t^-) 
 & i = 1, \ldots, m
}
\eeq
Equip $\calZ$ with the twisted conjugation $c_\calZ(a, x, y, t^+, t^-) = (\bar a, \bar y, \bar x,  \bar t^-, \bar t^+)$.
 Projection  provides an identification of real points
\beq
\xymatrix{
\calZ(\bbR) \simeq Z(\bbR) \times  B^+(\bbC) \simeq Z(\bbR)  \times  \bbC^m
}
\eeq

Introduce  the canonical section 
\beq
\xymatrix{
 \xi:B\ar[r] &  \calZ & 
 \xi(a, t^+, t^-) = (a, [1, a], [1, a], t^+, t^-) 
}
\eeq

We denote by $B'$ the open subset of $B$ consisting of 
$(a,t^+,t^-)$ with $a\neq 0$ and we define 
$\calZ':=\calZ\times_BB'$.

Note that  $\xi(a, t^+, t^-) = \sigma_i^\pm(a, t^+, t^-)$ if and only if   $t_i^\pm =a^{-1}$. 
Denote by $B_\circ \subset B$ the open complement of such coincidences.
Introduce  the base-change
\beq
\xymatrix{
p:\calZ_\circ = \calZ\times_{B} B_\circ\ar[r] & B_\circ 
}\eeq
and note, by construction, the tautological sections 
\beq
\xymatrix{
\sigma=(\sigma^+, \sigma^-):B_\circ \ar[r] & \calZ_\circ^{2m}
}
\eeq
 do not intersect the canonical section  
\beq
\xymatrix{
\xi:B_\circ \ar[r] & \calZ_\circ 
}\eeq

With the above choices fixed, we will study the real points of the  ind-stacks of quasimaps with their natural projections
\beq
\xymatrix{
q:\QM_{G, X}(\calZ/B , \sigma)_\bbR \ar[r] & B(\bbR) 
}
\eeq
\beq
\xymatrix{ 
r:\QM_{G, X}(\calZ_\circ/B_\circ , \sigma, \xi)_\bbR\ar[r] & B_\circ(\bbR) 
}
\eeq

\subsection{Complex groups}
We specialize here our prior constructions to the distinguished case of complex groups.

Recall $\delta = \theta \circ \eta = \eta \circ \theta$ denotes the Cartan conjugation of $G$ with compact real form $G_c$. 
Equip $G\times G$ with the swap involution $\sw(g,h)=(h, g)$ and 
the conjugation $ \sw_\delta(g,h)=(\delta(h),\delta(g))$. The fixed-point subgroup of $\sw$ is the diagonal  $G\subset G\times G$, and the corresponding symmetric space is  isomorphic to the group
$G\backslash(G\times G)\is G$.

\begin{lemma}\label{uniformization for real QM}
The uniformization morphisms 
\eqref{eq:qm almost unif real}, \eqref{eq:qm unif real} are isomorphisms:
\beq
\xymatrix{
G[\calZ',\sigma]_\bbR \bs \Gr_{G\times G,\calZ',\sigma,\bbR}\times_{B'}B'(\bbR)
\ar[r]^-\sim &  \QM_{G\times G,G}(\calZ'/B', \sigma)_\bbR  
}
\eeq
\beq
\xymatrix{
G[\calZ'_\circ,\sigma,\xi]_\bbR \bs\Gr_{G\times G,\calZ',\sigma,\bbR}\times_{B'}B'_\circ(\bbR)
\ar[r]^-\sim &  \QM_{G\times G,G}(\calZ'_\circ/B'_\circ, \sigma,\xi)_\bbR  
}
\eeq
\end{lemma}
\begin{proof}
Note that the above uniformization morphisms over a base point
$b\in B'$ (resp. $b\in B'_\circ$) are the (multipoint version) of the real or complex uniformization morphisms in \cite[Section 6.3]{CN1}.
Since the fixed-point subgroup $G\subset G\times G$ is connected and 
$H^1(\Gal(\bC/\bbR),G\times G)$ is trivial for the Galois-action given by $\sw_\delta$, it follows from 
 \cite[Remark 5.8 and Lemma 6.1]{CN1}
that any real bundles on the curve $\calZ_b(\bbR)\is\mathbb{RP}^1$ (associated to the real form $\sw_\delta$ of $G\times G$) admit either  real or complex uniformizations. The lemma follows by standard arguments.
\end{proof}


Next, the projection maps $\pr_1,\pr_2:G\times G\to G$, $\pr_i(g_1, g_2) = g_i$,
provide an isomorphism 
\beq\label{e:factorization}
\xymatrix{
\Gr_{G\times G,\calZ,\sigma}\ar[r]^-\sim & \Gr_{G,\calZ,\sigma}\times_B\Gr_{G,\calZ,\sigma}
}\eeq
The conjugations $c_\calZ$ of $\calZ$ and $\sw_\delta$ of $G \times G$ together  induce a conjugation of $\Gr_{G\times G,\calZ,\sigma}$ which,
 under the isomorphism~\eqref{e:factorization}, is given by the map on pairs of bundles with sections 
\beq
\xymatrix{
 (\mE,s,\mE',s')\ar@{|->}[r] & ( c_\calZ^*\mE'_\delta,c_\calZ^*(s'),
c_\calZ^*\mE_\delta,c_\calZ^*(s))
}
\eeq
Thus the isomorphism \eqref{e:factorization} followed by $\pr_1$
provides 
an isomorphism 
\beq\label{real}
\xymatrix{
\Gr_{G\times G,\calZ,\sigma,\bbR}\ar[r]^-\sim & \Gr_{G,\calZ,\sigma} \times_B B(\bbR)
}
\eeq
of real analytic spaces over $B(\bbR) \simeq \bbP^1(\bbR) \times_{\bbP^1(\bbC)} B^+(\bbC)$.

Thus the preceding lemma has the following consequence.

\begin{corollary}\label{c:uniformization for real QM}
There are natural isomorphisms
\beq
\xymatrix{
G[\calZ',\sigma]_\bbR\bs(\Gr_{G,\calZ',\sigma}\times_{B'} B'(\bbR)) 
\ar[r]^-\sim &   \QM_{G\times G,G}(\calZ'/B', \sigma)_\bbR  
}
\eeq
\beq
\xymatrix{
G[\calZ'_\circ,\sigma,\xi]_\bbR\bs(\Gr_{G,\calZ'_\circ,\sigma}\times_{B'_\circ} B'_\circ(\bbR))
\ar[r]^-\sim & 
  \QM_{G\times G,G}(\calZ'_\circ/B'_\circ, \sigma,\xi)_\bbR  
}
\eeq
\end{corollary}

Next, consider the natural map
\beq\label{e:extension}
\xymatrix{
 \Gr_{G,\calZ, \sigma^+}  \ar[r] &  \Gr_{G,\calZ, \sigma} 
}
\eeq
given by restricting
a trivialization of a $G$-bundle defined away from the section $\sigma^+$ to the complement of both sections  $\sigma^+$ and  $\sigma^-$.

We will find open loci where the map~\eqref{e:extension} induces an isomorphism.

First, consider the restriction of \eqref{e:extension} to the generic family
\beq\label{e:gen extension}
\xymatrix{
 \Gr_{G,\calZ', \sigma^+}  \ar[r] &  \Gr_{G,\calZ', \sigma} 
}
\eeq

\begin{prop} The map \eqref{e:gen extension}
induces isomorphisms
\beq
\xymatrix{
 G_c\bs \Gr_{G,\calZ', \sigma^+} \times_{B'} B'(\bbR) \ar[r]^-\sim & G[\calZ',\sigma]_\bbR\bs (\Gr_{G,\calZ',\sigma} \times_{B'} B'(\bbR))
}
\eeq
\beq
\xymatrix{
  \Gr_{G,\calZ'_\circ, \sigma^+}  \times_{B'_\circ} B'_\circ(\bbR) \ar[r]^-\sim & G[\calZ'_\circ,\sigma, \xi]_\bbR\bs (\Gr_{G,\calZ'_\circ,\sigma}\times_{B'_\circ} B'_\circ(\bbR))
}
\eeq
\end{prop}

\begin{proof}
It suffices to establish the second with its natural $G_c$-equivariance then glue to obtain the first.
Let $b=(a,t_1^+,...,t_n^+)\in B'_\circ(\bbR)$ and let $z_1,...,z_k\in\bC$ be $k$-distinct points 
such that $\{z_1,..,z_k\}=\{t_1^+,...,t_n^+\}$.
It follows from the standard factorization property of Beilinson-Drinfeld Grassmannian
and Theorem \ref{thm: factor}  that, over the based point $b$, the second map above can be identified with the map
\beq\label{key iso}
\prod_{i=1}^k\Gr_{z_i}\to\prod_{i=1}^k(\Gr_{z_i}\times\Gr_{\frac{1}{a^2\bar z_i}})\to\prod_{i=1}^k\Omega_{z_i,a^{-1}} G_c\backslash(\Gr_{z_i}\times\Gr_{\frac{1}{a^2\bar z_i}})
\eeq
where the first map is the left copy embedding sending
$\gamma$ to $(\gamma,e)$, where $e$ is the base point of 
$\Gr_{\frac{1}{a^2\bar z_i}}$, 
and the second map is the natural quotient map.
(Note that in terms of the local coordinate $z_i$ the complex conjugation on $\mathcal Z_b\is\mathbb P^1_{z_i}$
is given by $z_i\to \frac{1}{a^2\bar z_i}$.)
Thus the assertion follows from the fact that 
the based loop group $\Omega_{z_i,a^{-1}} G_c$
acts freely on $\Gr_{z_i}$ and $\Gr_{\frac{1}{a^2\bar z_i}}$.
\end{proof}

\begin{corollary}\label{c:gen unif}
There are natural isomorphisms
\beq
\xymatrix{
 G_c\bs \Gr_{G,\calZ', \sigma^+} \times_{B'} B'(\bbR)
\ar[r]^-\sim &   \QM_{G\times G,G}(\calZ'/B', \sigma)_\bbR  
}
\eeq
\beq\label{e:gen unif}
\xymatrix{
 \Gr_{G,\calZ'_\circ, \sigma^+}  \times_{B'_\circ} B'_\circ(\bbR) 
\ar[r]^-\sim & 
  \QM_{G\times G,G}(\calZ'_\circ/B'_\circ, \sigma,\xi)_\bbR  
}
\eeq
\end{corollary}

\quash{
\begin{remark}
The above isomorphisms for the generic family $\calZ'\to A'$ naturally extend to the universal family $\calU \to U$.
Thus there are natural isomorphisms
\beq
\xymatrix{
 G_c\bs \Gr_{G,\calU, \sigma^+} \times_{U} U(\bbR)
\ar[r]^-\sim &   \QM_{G\times G,G}(\calU/U, \sigma)_\bbR  
}
\eeq
\beq
\xymatrix{
 \Gr_{G,\calU_\circ, \sigma^+}  \times_{U_\circ} U_\circ(\bbR) 
\ar[r]^-\sim & 
  \QM_{G\times G,G}(\calU_\circ/U_\circ, \sigma,\xi)_\bbR  
}
\eeq
extending those of the corollary.
\end{remark}
}

Next, let 
$\Bun_G^0(\calZ/B)\subset \Bun_G(\calZ/B)$ denote the open sub-stack of 
a point $b\in B$ and a  
trivializable 
$G$-bundle on $\calZ_b$.
Denote by 
\beq
\xymatrix{
QM^0_{G\times G,G}(\calZ/B,\sigma)_\bbR \simeq G_c \backslash  \Gr_{G,\calZ,\sigma^+}^0
&
QM^0_{G\times G,G}(\calZ_\circ/B_\circ,\sigma,\xi)_\bbR \simeq \Gr_{G,\calZ_\circ^+,\sigma^+}^0
}
\eeq
 the 
base-changes
to $\Bun_G^0(\calZ/B)\subset \Bun_G(\calZ/B)$.
Consider the restriction of \eqref{e:extension} to the trivial bundle locus
\beq\label{e:triv extension}
\xymatrix{
 \Gr^0_{G,\calZ, \sigma^+}  \ar[r] &  \Gr^0_{G,\calZ, \sigma} 
}
\eeq

\begin{prop} The map \eqref{e:triv extension}
induces isomorphisms
\beq
\xymatrix{
 G_c\bs \Gr^0_{G,\calZ, \sigma^+} \times_{B} B(\bbR) \ar[r]^-\sim & G[\calZ,\sigma]_\bbR\bs (\Gr^0_{G,\calZ,\sigma} \times_{B} B(\bbR))
}
\eeq
\beq
\xymatrix{
  \Gr^0_{G,\calZ_\circ, \sigma^+}  \times_{B_\circ} B_\circ(\bbR) \ar[r]^-\sim & G[\calZ_\circ,\sigma, \xi]_\bbR\bs (\Gr^0_{G,\calZ_\circ,\sigma}\times_{B_\circ} B_\circ(\bbR))
}
\eeq
\end{prop}

\begin{proof}
By Proposition~\eqref{e:gen extension}, it suffices to establish the maps are isomorphisms at the special fiber $b=(a,t_1^+,...,t_n^+)\in B(\bbR)$
with $a=0$.
Moreover, it suffices to establish the second with its natural $G_c$-equivariance then glue to obtain the first.

Now it is elementary to see we have a natural isomorphism  at the special fiber
\beq\label{factorization at 0}
\xymatrix{
\Gr_{G,\calZ_\circ,\sigma}\times_{B_\circ} B_\circ(\bbR)|_{b}
\ar[r]^-\sim & 
\Gr_{G,\calZ_\circ,\sigma^+,b}\times \Gr_{G,\calZ_\circ,\sigma^-,b} 
}
\eeq
And similarly, we have a natural isomorphism of groups at the special fiber
\beq
\xymatrix{
G[\calZ_\circ,\sigma, \xi]_\bbR|_{b} \ar[r]^-\sim &  \Delta_\bbR \subset G[\bbP^1, \sigma^+,\xi,b ] \times G[\bbP^1, \sigma^-,\xi,b]
}
\eeq
where $\Delta_\bbR$ denotes the conjugate diagonal. 
Thus the assertion follows from the facts that the isomorphism~\eqref{factorization at 0} restricts to
an isomorphism 
\beq\label{open locus}
\xymatrix{
\Gr^0_{G,\calZ_\circ,\sigma}\times_{B_\circ} B_\circ(\bbR)|_{b}
\ar[r]^-\sim & 
\Gr^0_{G,\calZ_\circ,\sigma^+,b}\times  \Gr^0_{G,\calZ_\circ,\sigma^-,b} 
}
\eeq
and the action of $G[\bbP^1, \sigma^+, \xi,b]$
on $\Gr^0_{G,\calZ_\circ,\sigma^+,b}$ is free and transitive.
\end{proof}

\begin{corollary}\label{c:triv unif}
There are natural isomorphisms
\beq
\xymatrix{
  G_c\bs \Gr^0_{G,\calZ, \sigma^+} \times_{B} B(\bbR)
\ar[r]^-\sim &   \QM^0_{G\times G,G}(\calZ/B, \sigma)_\bbR  
}
\eeq
\beq\label{e:specical unif}
\xymatrix{
  \Gr^0_{G,\calZ_\circ, \sigma^+}  \times_{B_\circ} B_\circ(\bbR)\ar[r]^-\sim & 
  \QM^0_{G\times G,G}(\calZ_\circ/B_\circ, \sigma,\xi)_\bbR  
}
\eeq
\end{corollary}

\begin{remark}
The isomorphisms of Cor 5.5 and 5.8 coincide on the intersections of their domains.
\end{remark}

\subsection{Stratification}

Recall the Beilinson-Drinfeld grassmannian $\Gr^{(m)}_{}\to(\bP^1)^m$ of 
a $G$-bundle $\mE$ on $\bP^1$, a point $(z_1,...,z_m)\in(\bP^1)^m$, and a section 
$s:\bP^1_{}\setminus\{z_1,...,z_m\}\to\mE$.
We denote by
$\Gr^{(m),0}\subset\Gr^{(m)}$ the open subset consisting of $(\mE,z_1,...,z_m,s)$
such that $\mE$ is trivializable.

First, let us
stratify the base $(\bP^1)^m$ by coincidences among points. For any partition $\fp$ of the set $\{1, \ldots, m\}$, denote by $(\bP^1)^\fp \subset (\bP^1)^m$ the locus where $z_i=z_j$ if and only if $i$ and $j$ are in the same part of $\frakp$. This provides a Whitney stratification of $(\bP^1)^m$.

Next, for any partition $\frakp$ of the set $\{1, \ldots, m\}$, and any map $\lambda_\frakp:\frakp\to \Lambda_T^+$, denote by 
$
\calS^{\lambda_\frakp}_{} \subset \Gr^{(m)} 
$ (resp. $\calS^{\lambda_\frakp,0}_{} \subset \Gr^{(m),0} $)
the spherical stratum of a $G$-bundle $\mE$ (resp. a trivializable $G$-bundle $\mE$),
a point $(z_1,...,z_m) \in(\bbP^1)^\frakp$,
and a section $s:\bbP^1\setminus\{ z_1,\ldots, z_m\}\to\mE$ of modification type $\lambda_\fp$. 
This provides a Whitney stratification of $\Gr_{}^{(m)}$ (resp. $\Gr^{(m),0}$)
compatible with that of $(\bbP^1)^m$.

Note that 
the coordinate 
$x^{-1}$ provides an isomorphism
\beq\label{triv}
(\bbR\times\Gr^{(m)})\times_{\bbR\times(\bP^1)^m}B(\bbR)\is\Gr_{G,\calZ,\sigma^+}\times_BB(\bbR)
\eeq 
here we identify $B(\bbR)\stackrel{\eqref{B(R)}}\is\bbR\times \bC^m
\subset\bbR\times(\bP^1)^m$. 
The isomorphism above 
restricts to an isomorphism 
\beq\label{triv B}
(\bbR\times\Gr^{(m),0})\times_{\bbR\times(\bP^1)^m}B(\bbR)\is\Gr^0_{G,\calZ,\sigma^+}\times_BB(\bbR)
\eeq 
between the corresponding open loci 
and, by 
Corollary~\ref{c:gen unif} and~\ref{c:triv unif}, we obtain:
\beq\label{e:gen triv quasi}
(\bbR\times\Gr^{(m)})\times_{\bbR\times(\bP^1)^m}B'_\circ(\bbR)\is QM_{G\times G,G}(\calZ'_\circ/B'_\circ,\sigma,\xi)_\bbR,
\eeq
\beq\label{e:triv triv quasi}
(\bbR\times\Gr^{(m),0})\times_{\bbR\times(\bP^1)^m}B_\circ(\bbR)\is QM^0_{G\times G,G}(\calZ_\circ/B_\circ,\sigma,\xi)_\bbR.
\eeq

Let us equip $\bbR\times\Gr^{(m)}$ (resp. $\bbR\times\Gr^{(m),0}$) with the product stratification 
$\{\bbR\times\calS^{\lambda_\frakp}\}$ (resp. $\{\bbR\times\calS^{\lambda_\frakp,0}\}$) and 
transport the spherical strata $
\bbR\times\calS^{\lambda_\frakp} \subset \bbR\times\Gr^{(m)}_{} 
$ (resp. $
\bbR\times\calS^{\lambda_\frakp,0} \subset \bbR\times\Gr^{(m),0}_{} 
$) across  the isomorphisms  \eqref{e:gen triv quasi} and \eqref{e:triv triv quasi}
and denote the resulting strata by 
\beq
\xymatrix{
\calS^{\lambda_\frakp}_{\calZ, \bbR} \subset \QM_{G\times G,G}(\calZ'_\circ/B'_\circ, \sigma,\xi)_\bbR
&
(\text{resp.}\ \ \  \calS^{\lambda_\frakp,0}_{\calZ, \bbR} \subset \QM^0_{G\times G,G}(\calZ_\circ/B_\circ, \sigma, \xi)_\bbR).
}\eeq

Note that $\calS_{\calZ, \bbR}^{\lambda_\frakp, 0}$ and $\calS_{\calZ,\bbR}^{\lambda_\frakp, 0}$ are non-empty
if and only if the total coweight $|\lambda_\fp|\in \Lambda_T^+$, given by summing the values of $\lambda_\fp$ over the parts of $\fp$, in fact lies in $R^+_G \subset  \Lambda_T^+$.

\quash{
Let us summarize the above discussion in the following proposition

\begin{prop}
\begin{enumerate}
\item
The stratified map 
\[\xymatrix{
QM_{G\times G,G}(\calZ'_\circ/B'_\circ, \sigma,\xi)_\bbR\ar[r]&B_\circ'}
\]
where we equip $QM_{G\times G,G}(\calZ'_\circ/B'_\circ, \sigma,\xi)_\bbR$ and 
$B_\circ'$ with the stratifications $\{\calS_{\calZ,\bbR}^{\lambda_\frakp}\}$ and $\{(B_\circ')^{\frakp}\}$,
admits  a real analytic trivialization over $\bbR^\times$. 
\item
The stratified map
\[\xymatrix{
QM^0_{G\times G,G}(\calZ_\circ/B_\circ, \sigma,\xi)_\bbR\ar[r]&B_\circ}
\]
where we equip $QM^0_{G\times G,G}(\calZ_\circ/B_\circ, \sigma,\xi)_\bbR$ and 
$B_\circ$ with the stratifications $\{\calS_{\calZ,\bbR}^{\lambda_\frakp,0}\}$ and $\{B_\circ^\frakp\}$,
admits  a real analytic trivialization over $\bbR$.
\end{enumerate}
\end{prop}

}

\subsection{Symmetry} 
We continue here with the distinguished case of complex groups. We will describe an involution of quasi-maps.

\begin{definition}[Swap involution $\delta_\calZ$]
The conjugation $\delta \times \delta$ of $G\times G$ commutes with the conjugation $\sw_\delta$. Hence together with the conjugation 
$c_\calZ$ of $\calZ$ it
induces a fiberwise involution of $\QM_{G\times G,G}(\calZ/B,\sigma)_\bbR$ and 
$\QM_{G\times G,G}(\calZ_\circ/B_\circ,\sigma,\xi)_\bbR$ denoted by  $\delta_\calZ$.
\end{definition}

\quash{
\begin{definition}[Curve involution $\beta_\calZ$]
 The family of curves $\calZ\to B$ has a natural fiberwise involution 
$\beta_\calZ$ defined by  $\beta_\calZ(a, x, y, t^+, t^-) = (a, y, x, t^-, t^+)$.
It induces a fiberwise involution of   $\QM_{G\times G,G}(\calZ/B,\sigma)_\bbR$
and $\QM_{G\times G,G}(\calZ_\circ/B_\circ,\sigma, \xi)_\bbR$ denoted by $\beta_\calZ$. 
\end{definition}

The two involutions evidently commute.

\begin{lemma}
The involutions $(\delta_\calZ, \beta_\calZ)$ provide a $\bbZ/2 \times \bbZ/2$-action.
\end{lemma}
}

In the remainder of this section, we will give concrete descriptions of how the involution
$\delta_\calZ$
 act at the generic and special fibers in terms of our prior uniformizations.

\quash{
Consider the natural map
\[v:\Gr_{G,\calZ_\circ,\sigma^+}\stackrel{\eqref{e:extension}}\lra\Gr_{G,\calZ_\circ,\sigma}\lra QM_{G\times G,G}(\calZ_\circ/B_\circ,\sigma,\xi)_\bbR\]
where the last map is the uniformization morphism.
For any $b=(a,t^+_1,...,t^+_m)\in B_\circ(\bbR)\subset B(\bbR)\is\bbR\times\bC^m$
let 
\[v_b:\Gr_{G,\calZ_\circ,\sigma^+,b}\to QM_{G\times G,G}(\calZ_\circ/B_\circ,\sigma,\xi,b)_\bbR\]
be the base change of $v$ to $b$.
By Corollary \ref{} and \ref{}, the map above induces isomorphisms

\beq\label{a neq 0}
v_b:\Gr_{G,\calZ_\circ,\sigma^+,b}\is QM_{G\times G,G}(\calZ_\circ/B_\circ,\sigma,\xi,b)_\bbR,\ \ \text{if}\ \ b=(a\neq0,t^+_1,...,t^+_m)\in B_\circ'(\bbR),
\eeq

\beq\label{a=0}
v_b:\Gr^0_{G,\calZ_\circ,\sigma^+,b}\is QM^0_{G\times G,G}(\calZ_\circ/B_\circ,\sigma,\xi,b)_\bbR,\ \ \text{if}\ \ b=(0,t^+_1,...,t^+_m).
\eeq
}

\subsubsection{Generic fiber}
Let
$b=(a\neq 0,t^+_1,...,t^+_m)\in B'_\circ(\bbR)$.
Let $z_1,...,z_k\in\bC$ be the distinct $k$ points such that there is an equality of sets
$\{z_1,...,z_k\}=\{t^+_1,...,t^+_m\}$
and consider the isomorphism
\beq\label{gen v_b}
v_b:\Gr_{z_1}\times\cdot\cdot\cdot\times\Gr_{z_k}\is
(\bbR\times\Gr^{(m)})|_{b}\stackrel{\eqref{triv}}\is
\Gr_{G,\calZ,\sigma^+,b}\stackrel{\eqref{e:gen unif}}\is QM_{G\times G,G}(\calZ_\circ/B_\circ,\sigma,\xi,b)_\bbR.
\eeq
Here the first map is the factorization isomorphism.
For $i=1, \ldots, k$, we define 
\[v_{i}:\Gr_{z_i}\hookrightarrow\Gr_{z_1}\times\cdot\cdot\cdot\times\Gr_{z_k}\stackrel{v_b}\is QM_{G\times G,G}(\calZ_\circ/B_\circ,\sigma,\xi,b)_\bbR\]

\begin{prop}\label{action at gen fiber}

For $i=1, \ldots, k$,  the involution $\delta_\calZ$ satisfies:
\begin{enumerate}
\item
When $z_i\in S^1_{a^{-1}}=\{|z|=|a^{-1}|\}\subset\bP^1$, we have
\beq\label{v and delta}
\delta_\calZ\circ v_i=
v_i\circ\delta
\eeq
where 
$\delta$ is the conjugation on $\Gr_{z_i}$ induced by the conjugations 
$c(z)=\frac{1}{a^2\bar z}$ on $\bP^1$ and $\delta$ on $G$.

\item

When $z_i\in\bbC\setminus S^1_{a^{-1}}$, we have
\beq\label{v and delta}
\xymatrix{
\delta_\calZ\circ v_i=
v_i\circ\on{inv}
}
\eeq
where 
$\on{inv}(g)(z) = g(z)^{-1}$ is the group-inverse on 
$\Omega_{z_i,a^{-1}}G_c\is\Gr_{z_i}$.

\end{enumerate}
\end{prop}

\begin{proof}
We first note that 
the  isomorphism \[\prod_{i=1}^k
\Omega_{z_i,a^{-1}}G_c\backslash(\Gr_{z_i}\times\Gr_{c(z_i)})\is
G[\calZ'_\circ,\sigma, \xi,b]_\bbR\bs \Gr_{G,\calZ'_\circ,\sigma,b}\is QM_{G\times G,G}(\calZ_\circ/B_\circ,\sigma,\xi,b)_\bbR\]
 intertwines the 
conjugation $\delta_\calZ$ with 
the one $\delta_{\calZ}'$ on the product $\prod_{i=1}^k\Omega_{z_i,a^{-1}}G_c\backslash(\Gr_{z_i}\times\Gr_{c(z_i)})$ 
given by 
\[\delta_{\calZ}'(\gamma_{i,1},\gamma_{i,2})= (\gamma'_{i,1},\gamma'_{i,2}),\ \ \ i=1,...,k\] where 
$\gamma_{i,j}'=\delta(\gamma_{i,j})$, $j=1,2$, if $z_i=c(z_i)\in S^1_{a^{-1}}$, otherwise 
$\gamma_{i,1}'=\delta(\gamma_{i,2})$ 
and $\gamma_{i,2}'=\delta(\gamma_{i,1})$
if $z_i\neq c(z_i)$. 
Here $\delta$ is the map $\delta:\Gr_{z_i}\to\Gr_{c(z_i)}$ induced by the complex conjugation
$c(z)=\frac{1}{a^2\bar z}$ on $\bP^1$ and $\delta$ on $G$.

On the other hand, 
a direct computation shows that the isomorphism
\[\prod_{i=1}^k\Gr_{z_i}\is\prod_{i=1}^k\Omega_{z_i,a^{-1}}G_c\backslash(\Gr_{z_i}\times\Gr_{c(z_i)})\]
induced by the left copy embedding  $\gamma\to (\gamma,e)$  intertwines the 
map $\delta_\calZ'$ above with the one $\delta''_{\calZ}$ on $\Gr_{z_i}$ given by $\delta_\calZ''(\gamma)=\delta(\gamma)$  when $z_i=c(z_i)$
and $\delta_\calZ''(\gamma)=\gamma^{-1}$ is
 the group-inverse on 
$\Omega_{z_i,a^{-1}}G_c\is\Gr_{z_i}$ when $z_i\neq c(z_i)$.

To deduce the proposition, we note that the map $v_b$ in~\eqref{gen v_b}
 is equal to the composition of the above two isomorphisms.
\end{proof}

\subsubsection{Special fiber}
Let $b=(0,t_1^+,...,t_m^+)\in B_\circ(\bbR)$.
Consider the map
\beq\label{triv v_b}
v_b:G[\bP^1\setminus\{t_1^+,...,t_m^+\}]_{\infty\to e}\is(\bbR\times\Gr^{(m),0})|_{b}\stackrel{\eqref{triv B}}\is 
\Gr^0_{G,\calZ,\sigma^+,b}\stackrel{\eqref{e:specical unif}}\is QM^0_{G\times G,G}(\calZ_\circ/B_\circ,\sigma,\xi,b)_\bbR.
\eeq

\begin{prop}\label{action at special fiber}
The involution $\delta_\calZ$ satisfies
\[\delta_\calZ\circ v_b=v_b\circ\on{inv}\]
here $\on{inv}(g(z))=g(z)^{-1}$ is the group inverse on $G[\bP^1\setminus\{t_1^+,...,t_m^+\}]_{\infty\to e}$.
\end{prop}

\begin{proof}
We first note that the isomorphism
\[
G[\bP^1\setminus\{t_1^+,...,t_m^+\}]_{\infty\to e}\backslash(\Gr^0_{G,\calZ,\sigma^+,b}
\times\Gr^0_{G,\calZ,\sigma^-,b})\is G[\calZ'_\circ,\sigma, \xi,b]_\bbR\bs \Gr^0_{G,\calZ'_\circ,\sigma,b}\is QM^0_{G\times G,G}(\calZ_\circ/B_\circ,\sigma,\xi,b)_\bbR\]
intertwines the involution $\delta_\calZ$ with the one $\delta_\calZ'$
on $G[\bP^1\setminus\{t_1^+,...,t_m^+\}]_{\infty\to e}\backslash(\Gr^0_{G,\calZ,\sigma^+,b}
\times\Gr^0_{G,\calZ,\sigma^-,b})$ given 
by $\delta_\calZ'(\gamma_1,\gamma_2)=(\delta(\gamma_2),\delta(\gamma_1))$
(note that the group $G[\bP^1\setminus\{t_1^+,...,t_m^+\}]_{\infty\to e}$
acts on the product via the conjugate diagonal embedding 
$\gamma\to(\gamma,\delta(\gamma))$).
On the other hand, the isomorphism 
\[G[\bP^1\setminus\{t_1^+,...,t_m^+\}]_{\infty\to e}\is
\Gr^0_{G,\calZ,\sigma^+,b}
\is
G[\bP^1\setminus\{t_1^+,...,t_m^+\}]_{\infty\to e}\backslash(\Gr^0_{G,\calZ,\sigma^+,b}
\times\Gr^0_{G,\calZ,\sigma^-,b})\]
where the fist map is the action map and the second map is 
induced by the left copy embedding intertwines the 
involution $\delta_\calZ'$ with 
the group inverse on $G[\bP^1\setminus\{t_1^+,...,t_m^+\}]_{\infty\to e}$.
To deduce the proposition, we observe that the map 
$v_b$ is equal to the composition of the above two isomorphisms.
\end{proof}


\subsection{Compatibility with stratifications}
We have the following compatibility of the involution $\delta_\calZ$ with stratifications.
Let $w_0\in W$ denote the longest element of the Weyl group. 
For any partition $\fp$ of the set $\{1, \ldots, m\}$, and map $\lambda_\fp:\fp\to \Lambda_T^+$, 
we set $-w_0(\lambda_\fp) = -w_0 \circ \lambda_\fp$.

\begin{lemma}\label{stable under eta}
\begin{enumerate}
\item
The involution $\delta_\calZ$  preserves the open subspace
\beq
\QM^0_{G\times G,G}(\calZ_\circ/B_\circ,\sigma,\xi)_\bbR
\subset \QM_{G\times G,G}(\calZ_\circ/B_\circ,\sigma,\xi)_\bbR
\eeq

\item
The involution $\delta_\calZ$ restricts to a map on spherical strata
\beq
\xymatrix{
\calS_{\calZ,\bbR}^{\lambda_\fp}\ar[r]  & \calS_{\calZ,\bbR}^{-w_0(\lambda_\fp)}
}\eeq
within $\QM_{G\times G,G}(\calZ'_\circ/B'_\circ,\sigma,\xi)_\bbR$,
and also the spherical strata 
\beq
\xymatrix{
\calS_{\calZ,\bbR}^{\lambda_\fp, 0}\ar[r]  & \calS_{\calZ, \bbR}^{-w_0(\lambda_\fp), 0}
}\eeq
within $\QM^0_{G\times G,G}(\calZ_\circ/B_\circ,\sigma,\xi)_\bbR$.
\end{enumerate}
\end{lemma}
\begin{proof}
The claim follows from the facts that 
the conjugation $\delta$ of $\Gr$ maps 
 $S^\lambda$ (resp. $\Gr^0$) to
$S^{-w_0(\lambda)}$
(resp.  $\Gr^0$), and the involution $\on{inv}$ of $\Omega G_c$ maps 
  $S^\lambda$  to 
$S^{-w_0(\lambda)}$.
\end{proof}

\subsection{The involution $\eta_\calZ$}\label{eta_Z}
Recall $\theta$ denotes the Cartan involution of $G$ with fixed-point subgroup $K$,
and $\delta = \theta \circ \eta = \eta \circ \theta$ the Cartan conjugation of $G$ with compact real form $G_c$. 
Since $\eta$ and $\theta$ commute with $\delta$,
the conjugation   $\eta \times \eta$ and involution   $\theta \times \theta$ of $G\times G$ define involutions of the real moduli
$QM_{G\times G,G}(\calZ_\circ/B_\circ,\sigma,\xi)_\bbR$ 
 of rigidified quasi-maps  which we denote respectively by $\eta_\calZ$ and $\theta_\calZ$.

Propositions~\ref{action at gen fiber} and \ref{action at special fiber},
and Lemma \ref{stable under eta}
immediately imply:

\begin{proposition}\label{Z/2 action}

The involution $\eta_\calZ$ on $\QM_{G\times G,G}(\calZ_\circ/B_\circ,\sigma,\xi)_\bbR$ satisfies:

\begin{enumerate}
\item The involution commutes with the  natural $G_c$-action
and preserves the open subfamilies 
$\QM_{G\times G,G}(\calZ'_\circ/B'_\circ,\sigma,\xi)_\bbR$
and $\QM^0_{G\times G,G}(\calZ_\circ/B_\circ,\sigma,\xi)_\bbR$
and their spherical stratifications
$\{\calS^{\lambda_\frakp}_{\calZ,\bbR}\}$
and $\{\calS^{\lambda_\frakp,0}_{\calZ,\bbR}\}$.

\item
At $b=(a\neq 0,t^+_1,...,t^+_m)\in B_\circ(\bbR)$, the isomorphism
\[\xymatrix{
v_b:\Gr_{z_1}\times\cdot\cdot\cdot\times\Gr_{z_k}\ar[r]^-\sim&\QM_{G\times G,G}(\calZ_\circ/B_\circ,\sigma,\xi,b)_\bbR}\] of \eqref{gen v_b}
intertwines the involution $\delta_\calZ$ with the involution on $\Gr_{z_1}\times\cdot\cdot\cdot\times\Gr_{z_k}$ given by $(\gamma_1,...,\gamma_k)\to (\gamma_1',...,\gamma_k')$, where 
$\gamma'_i=\eta(\gamma_i)$ if $z_i\in S^1_{a^{-1}}$ and 
$\eta$ is the conjugation on $\Gr_{z_i}$ induced by 
$c(z)=\frac{1}{a^2\bar z}$ on $\bP^1$ and $\eta$ on $G$, otherwise 
$\gamma'_i=\on{inv}\circ\theta(\gamma_i)$ if $z_i\in\bC\setminus S^1_{a^{-1}}$ and 
$\on{inv}\circ\theta$ is the involution on $\Omega_{z_i,a^{-1}}G_c\is\Gr_{z_i}$.

\item 
At $b=(0,t_1^+,,,.t_m^+)\in B_\circ(\bbR)$, the isomorphism
\[v_b:\xymatrix{G[\bP^1\setminus\{t_1^+,,,.t_m^+\}]_{\infty\to e}\ar[r]^-\sim&QM^0_{G\times G,G}(\calZ_\circ/B_\circ,\sigma,\xi,b)_\bbR}\]
of \eqref{triv v_b} intertwines the involution $\eta_\calZ$ with the involution 
$\on{inv}\circ\theta(\gamma)=\theta(\gamma)^{-1}$ on 
$G[\bP^1\setminus\{t_1^+,,,.t_m^+\}]_{\infty\to e}$.

\end{enumerate}
\end{proposition}

\section{Stratified homeomorphisms}


\subsection{Trivializations}

Set $\mathbb B=[0,1]\times\bbR^m$ and $\mathbb B'=(0,1]\times\bbR^m$.

\begin{definition}
An embedding 
\[\xymatrix{\zeta:\mathbb B=[0,1]\times\bbR^m\ar[r]&B_\circ(\bbR)}\subset\bbR\times\bC^m\]
is called \emph{admissible} if it is given by  
$\zeta(a,z_1,...,z_m)=(a,f_a(z_1),...,f_a(z_m))$
where $f:[0,1]\to\on{Aut}(\bP^1)$
is 
a real analytic map 
satisfying 
$f_0(z)=z$ and $f_1(z)=\frac{z-i}{z+i}$.
\end{definition}

\begin{remark}
Each admissible $\zeta$ defines a one-parameter family of embeddings 
$\zeta_a:=\zeta|_{a\times\bbR^m}:\bbR^m\to\bbC^m$,\ $a\in[0,1]$ satisfying 
$\zeta_0(\bbR^m)=\bbR^m$ and $\zeta_1(\bbR^m)=(S^1\setminus\{1\})^m$.
\end{remark}

\begin{example}
Consider the map $f:[0,1]\to\on{Aut}(\bP^1)$ 
given by $f_a(z)=\frac{z-ai}{az+ai+(1-a)}$.
A direct computation shows that 
$f_a(z)\neq a^{-1}$ for all $a\in\bbR$,
and 
$f_0(z)=z$, $f_1(z)=\frac{z-i}{z+i}$, and hence 
the corresponding embedding 
$\zeta:\mathbb B\to B_\circ(\bbR)$ is admissible.

\end{example}

Let $\zeta:\mathbb B\to B_\circ(\bbR)$ be an admissible embedding.
Consider the following base changes
\[\xymatrix{QM_{G\times G,G}(\calZ/\mathbb B,\sigma,\xi)_\bbR:=
QM_{G\times G,G}(\calZ_\circ/B_\circ,\sigma,\xi)_\bbR\times_{B_\circ(\bbR)}\mathbb B\ar[r]&\mathbb B
\\
QM^0_{G\times G,G}(\calZ/\mathbb B,\sigma,\xi)_\bbR:=
QM^0_{G\times G,G}(\calZ_\circ/B_\circ,\sigma,\xi)_\bbR\times_{B_\circ(\bbR)}\mathbb B\ar[r]&\mathbb B
\\
QM_{G\times G,G}(\calZ'/\mathbb B',\sigma,\xi)_\bbR:=QM_{G\times G,G}(\calZ'_\circ/B'_\circ,\sigma,\xi)_\bbR\times_{B'_\circ(\bbR)}\mathbb B'\ar[r]&\mathbb B'}
\]
Then \eqref{e:gen triv quasi} and \eqref{e:triv triv quasi} 
restrict to isomorphisms 

\beq\label{e:gen iso over B}
\xymatrix{
(\bbR\times\Gr^{(m)})\times_{\bbR\times(\bP^1)^m}\mathbb B'\ar[r]^-\sim&QM_{G\times G,G}(\calZ'/\mathbb B',\sigma,\xi)_\bbR},
\eeq
\beq\label{e:triv iso over B}
\xymatrix{
(\bbR\times\Gr^{(m),0})\times_{\bbR\times(\bP^1)^m}\mathbb B\ar[r]^-\sim&QM^0_{G\times G,G}(\calZ/\mathbb B,\sigma,\xi)_\bbR}.
\eeq

Note that there is an isomorphism over $[0,1]$
\[\xymatrix{[0,1]\times(\Gr^{(m)}\times_{(\bP^1)^m}\bbR^m)\ar[r]^-\sim&(\bbR\times\Gr^{(m)})\times_{\bbR\times(\bP^1)^m}\mathbb B}\]
given by
\[(a,(\mE,z_1,...,z_m,s))=(a,(\alpha_a)_*\mE,\alpha_a(z_1),...,\alpha_a(z_m),(\alpha_a)_*s)\]
and in view of  \eqref{e:gen iso over B}, \eqref{e:triv iso over B}, we obtain 

\begin{prop}\label{trivialization}
Each admissible embedding 
$\zeta:\mathbb B\to B_\circ(\bbR)$ induces isomorphisms:
\beq\label{}
\xymatrix{
(0,1]\times(\Gr^{(m)}\times_{(\bP^1)^m}\bbR^m)\ar[r]^-\sim&QM_{G\times G,G}(\calZ'/\mathbb B',\sigma,\xi)_\bbR},
\eeq
\beq\label{}
\xymatrix{
[0,1]\times(\Gr^{(m),0}\times_{(\bP^1)^m}\bbR^m)\ar[r]^-\sim&QM^0_{G\times G,G}(\calZ/\mathbb B,\sigma,\xi)_\bbR}.
\eeq
\end{prop}
Note that the 
isomorphisms above coincide on the intersections of their domains.

\subsection{Families of involutions}

Let $p:\calG\to\bbR^m$ be the group ind-scheme over $\bbR^m$ whose fiber 
over $(z_1,...,z_m)\in\bbR^m$ is the group ind-scheme $G[\bP^1\setminus\{z_1,...,z_m\}]_{\infty\to e}$
of maps $\gamma:\bP^1\setminus\{z_1,...,z_m\}\to G$ such that $\gamma(\infty)=e$.
The acton of $\calG$ on the base point of 
$\Gr^{(m)}\times_{(\bP^1)^m}\bbR^m$ defines an isomorphism 
\beq
\xymatrix{\calG\ar[r]^-\sim&\Gr^{(m),0}\times_{(\bP^1)^m}\bbR^m}
\eeq
Consider the transported strata 
$\calG^{\lambda_\frakp}:=\calS^{\lambda_\frakp,0}\cap\calG$ of $\calG$ and equip 
$[0,1]\times\calG$ with the product stratification $\{[0,1]\times\calG^{\lambda_\frakp}\}$.
By Proposition \ref{trivialization}, we obtain 
a stratified isomorphism 
\beq\label{p_zeta}
p_\zeta:
\xymatrix{[0,1]\times\calG\ar[r]^-\sim&QM^0_{G\times G,G}(\calZ/\mathbb B,\sigma,\xi)_\bbR.}
\eeq

Recall 
the involution $\eta_\calZ$ on $QM^0_{G\times G,G}(\calZ_\circ/B_\circ,\sigma,\xi)_\bbR$
of Section \ref{eta_Z}. 
Via the isomorphism $p_\zeta$, 
the involution $\eta_\calZ$ 
gives rise to a family of 
involutions 
\beq\label{involution G}
\xymatrix{\alpha_a:\calG\ar[r]&\calG,\ \ \ a\in[0,1],}
\eeq
where $\alpha_a(\gamma):=p_\zeta^{-1}\circ\eta_\calZ\circ p_\zeta(\{a\}\times\gamma)$.
The following proposition follows from
Proposition \ref{Z/2 action}:

\begin{thm}\label{family of involutions}
The family of involutions $\alpha_a:\calG\lra\calG,\ a\in[0,1]$ satisfy the following:
\begin{enumerate}
\item We have $q\circ\alpha_a=q:\calG\to\bbR^m$.
\item $\alpha_a$ is $G_c$-equivariant and preserves the stratification $\{\calG^{\lambda_\frakp}\}$.
\item At $a=0$, we have $\alpha_0(\gamma(z))=\theta(\gamma(z))^{-1}$.
\item At $a=1$, we have $\alpha_1(\gamma(z))=\eta(\gamma(\bar z))$.

\end{enumerate}
\end{thm}
\begin{proof}
Part (1), (2), (3) is clear. Part (4) follows from the fact 
that the automorphism $f_1(z)=\frac{z-i}{z+i}$ satisfies
$f_1(\bar z)=\frac{\bar z-i}{\bar z+i}=(\overline{f_1(z)})^{-1}$.
\end{proof}

\subsection{Trivializations of fixed-points}\label{ss: fixed-points}
Our aim is to trivialize the fixed-point of the family involution $\beta_a$.
To that end, we will invoke the following lemma:

\begin{lemma}\label{fixed points}
Let $I\subset\bbR$ be an interval.
Let $M\to I$ and $N\to I$ 
be two stratified real analytic submersions of  
real analytic Whitney stratified ind-varieties $M$ and $N$ (where $I$ is equipped with the trivial stratification). Let $f:M\to N$ be a Thom map.
\begin{enumerate}
\item
Assume there is 
a compact group $H\times \bZ/2$ acting real analytically on $M$ such that 
the action preserves the stratifications 
and $f$ is 
$H\times\bZ/2$-invariant (where $H\times\bZ/2$ acts trivially on $N$). 
Then the 
$\bZ/2$-fixed-point ind-variety $M^{\bZ/2}$ is Whitney stratified by the 
fixed-points of the strata and the induced map 
$f^{\bZ/2}:M^{\bZ/2}\to N$ is an $H$-equivariant Thom map.
\item
Assume further that 
$f$ is ind-proper and 
there is an $H$-equivariant  stratified trivialization of 
$f:M\to N$ over $I$, 
that is, there are stratified preserving homeomorphisms 
$h_M$ and $h_N$
fitting into a commutative diagram
\[\xymatrix{M\ar[r]^{h_M\ \ \ }\ar[d]^f&I\times M_0\ar[d]^{\on{id}\times f_0}\\
N\ar[r]^{h_N\ \ \ }&I\times N_0
}\]
that are real analytic on each stratum.
Then there is an
$H$-equivariant  stratified trivialization of $f^{\bZ/2}:M^{\bZ/2}\to N$ that is real analytic on each 
stratum.
\end{enumerate}
\end{lemma}
\begin{proof}
Part (1) is proved in \cite[Lemma 4.5.1]{N}. 
For part (2), the $H$-equivariant stratified trivialization of $f:M\to N$
provides a horizontal lift 
of the coordinate vector field 
$\partial_t$ on $I$ to 
a continuous  $H$-invariant vector field 
$v$ on $M$ that is tangent to and real analytic  along each stratum.
Let $w$ be the 
average of $v$ with respect to the $\bZ/2\bZ$-action.
As $f$ is ind-proper and the $\bZ/2\bZ$-action is real analytic, the vector field $w$ is complete and the integral curves of 
$w$ define an $H$-equivariant stratified trivialization of $f^{\bZ/2}:M^{\bZ/2}\to N$ over $[0,1]$
that is real analytic along each 
stratum.
\end{proof}

Now let us apply the above lemma to 
the map 

\[\xymatrix{QM_{G\times G,G}(\calZ'/\mathbb B',\sigma,\xi)_\bbR\ar[r]&\mathbb B'
\\
(\text{resp.}\ \ QM^0_{G\times G,G}(\calZ/\mathbb B,\sigma,\xi)_\bbR\ar[r]&\mathbb B)}\]
with the stratifications 
$\{\calS^{\lambda_\frakp}_{\calZ,\bbR,\mathbb B'}\}$ and $\{(\mathbb B')^\fp=(0,1]\times\bbR^\fp\}$
(resp. $\{\calS^{\lambda_\frakp,0}_{\calZ,\bbR,\mathbb B}\}$ and $\{(\mathbb B^\fp=[0,1]\times\bbR^\fp\}$). We will consider the $H\times\bZ/2$-action given by
$K_c\times\langle\eta_\calZ\rangle$.

\begin{proposition}\label{Whitney}
\begin{enumerate}
\item
There is a $K_c$-equivariant topological trivialization of 
the fixed-points of $\eta_\calZ$ of the map \[QM_{G\times G,G}(\calZ'/\mathbb B',\sigma,\xi)_\bbR\lra\mathbb B'\] over $I=(0,1]$. 


\item
There is a $K_c$-equivariant topological trivialization of 
the fixed-points of $\eta_\calZ$ of the map \[QM^0_{G\times G,G}(\calZ/\mathbb B,\sigma,\xi)_\bbR\lra\mathbb B\] over $I=[0,1]$.

\end{enumerate}

\end{proposition}
\begin{proof}
For part (1), by Proposition \ref{trivialization}, there is a $K_c$-equivariant stratified trivialization of 
\beq\label{part 1}
\xymatrix{
\QM_{G\times G,G}(\calZ'/\mathbb B',\sigma,\xi)_\bbR \ar[r] &  \mathbb B'=(0,1]\times\bbR^m
}
\eeq
Applying Lemma~\ref{fixed points} with $H\times\bZ/2=K_c\times\langle\eta_\calZ\rangle$, we obtain
part (1).  

For part (2), by
 Proposition \ref{trivialization},
there is a $K_c$-equivariant stratified trivialization of 
\beq\label{part 2}
\xymatrix{
\QM^0_{G\times G,G}(\calZ'/\mathbb B',\sigma,\xi)_\bbR \ar[r] &  \mathbb B=[0,1]\times\bbR^m
}
\eeq
Following the proof of Lemma \ref{fixed points}, consider the averaged vector field $w$ 
with respect to the $\bZ/2\bZ$-action given by 
$\langle\eta_\calZ\rangle$.
We claim that $w$ is complete, hence 
the integral curves of 
$w$ provide the desired trivialization. 

To prove the claim, observe that, over the open locus $\mathbb B'\subset\mathbb B$
in the base,
 the 
$K_c$-equivariant trivialization in \eqref{part 2} 
extends to a $K_c$-equivariant trivialization of the ind-proper family 
in \eqref{part 1}.

Thus for any $b=(a,z_1,...,z_m)\in\mathbb B'$,
any integral curve $p(t)$ for $w$ with initial point 
 $p(a)\in \QM^0_{G\times G,G}(\calZ'/\mathbb B',\sigma,\xi,b)_\bbR$ exists for $t\geq a$.
Together with the local existence of integral curves with initial point in the special fiber 
$\QM^0_{G\times G,G}(\calZ'/\mathbb B',\sigma,\xi,b)_\bbR$, $b=(0,z_1,...,z_m)$, this 
implies $p(t)$ exists for all $t \in [0,1]$. Hence $w$ is complete 
and we have proved the claim.
\end{proof}

\subsection{Real and symmetric 
spherical strata}
Let us summarize here the results obtained by the preceding considerations. 

Recall we write $\calG\is\Gr^{(m),0}\times_{(\bP^1)^m}\bbR^m\ra\bbR^m$ for the group ind-scheme
over $\bbR^m$ whose fiber over 
$(z_1,...,z_m)\in\bbR^m$ is the group 
of maps $\gamma: \bbP^1\setminus \{z_1,...,z_m\} \to G$ such that
$\gamma(\infty) = e$.
Recall the transported spherical strata
$\calG^{\lambda_\frakp}=\calG\cap\calS^{\lambda_\frakp}$.

The isomorphism $p_\zeta:[0,1]\times\calG\is\QM^0_{G\times G,G}(\calZ'/\mathbb B',\sigma,\xi)_\bbR$
in
\eqref{p_zeta} together with
Proposition \ref{family of involutions} and \ref{Whitney} immediately imply:

\begin{thm}\label{main thm for Gr}

There is a $K_c$-equivariant stratified homeomorphism 
between the fixed-points of $\eta$ and $\on{inv}\circ\theta$
on $\calG$ compatibile with projections to $\bbR^m$.
The homeomorphism restricts to a $K_c$-equivariant real analytic isomorphism between
the fixed-points of $\eta$ and $\on{inv}\circ\theta$
on $\calG^{\lambda_{\frakp}}$.
\end{thm}

Observe that the fixed-points $\calG^\eta $ coincide with the group ind-scheme
$\calG_\bbR\ra\bbR^m$ of a point $(z_1,...,z_m)\in\bbR^m$ and a
map $\gamma: \bbP^1\setminus \{z_1,...,z_m\} \to G$, such that $\gamma(\bbP^1(\bbR) \setminus \{z_1,...,z_m\}) \subset G_\bbR$ and 
$\gamma(\infty) = e$.
Denote by $\calG^{\lambda_\frakp}_{\bbR}=\calG^{\lambda_\frakp}\cap\calG_\bbR$ its spherical strata. 
Similarly,  observe that the fixed-points  $(\calG)^{\on{inv}\circ \theta}$ coincides with the space $\calX\ra\bbR^m$ of a point $(z_1,...,z_m)\in\bbR^m$ and 
a map 
$\gamma: \bbP^1\setminus \{z_1,...,z_m\} \to X \subset G$ such that 
$\gamma(\infty) = e$.
Denote by $\calX^{\lambda_\frakp}=\calG^{\lambda_\frakp}\cap\calX$ its spherical strata.

We can restate the above theorem in the form:


\begin{thm}\label{G=X}
There is a $K_c$-equivariant stratified homeomorphism 
\begin{equation}\label{eq: homeo gr}
\xymatrix{
\calG_\bbR\ar[r]^-\sim& \calX
}
\end{equation}
fitting into the diagram 
\[\xymatrix{\calG_\bbR\ar[r]^-\sim\ar[d]&\calX\ar[d]
\\\bbR^m\ar[r]^{\on{id}}&\bbR^m}\]
that restricts to real analytic isomorphisms on strata 
\begin{equation}\label{eq: homeo gr strat}
\xymatrix{
\calG^{\lambda_\frakp}_{\bbR} \ar[r]^-\sim& \calX^{\lambda_\frakp}
}
\end{equation}
where $\lambda_\frakp:\frakp\to \mL^+$ 
(for the subset $\mL^+\subset\Lambda_A^+$ defined in Section 2.1) with $|\lambda_\frakp| \in R_G^+$.
\end{thm}

\begin{remark}
Since the stratum
$\calG^{\lambda_\frakp}_{\bbR}=S_\bbR^{\lambda_\frakp}\cap\calG_\bbR$ is the 
intersection of the real spherical stratum $S_\bbR^{\lambda_\frakp}$
in the real 
Beilinson-Drinfeld grassmannian $\Gr_\bbR^{(m)}$
with 
 the open cospherical stratum $\calG_\bbR$, the above theorem implies that
 the singularities of the closures of the real spherical strata $\overline {S_\bbR^{\lambda_\frakp}}$
 are \emph{locally}  homeomorphic to 
 complex algebraic varieties.
 However, one should note that 
the closures $\overline {S_\bbR^{\lambda_\frakp}}$ are not  in general 
\emph{globally}
homeomorphic to
complex projective varieties. For example, for the rank one group $G_\bbR = \SL_2(\bbH)$, and  the generator $\lambda\in \Lambda^+_A$ (i.e., $m=1$ and
$\lambda=\lambda^\fp:\fp=\{1\}\to\Lambda^+_A$), one finds that $\overline{S^{\lambda}_\bbR} \subset \Gr_\bbR$ is  the one-point compactification of the cotangent bundle $T^* \bbH\bbP^1$ of the quaternionic projective line. Thus its intersection cohomology Poincar\'e polynomial is $1 + t^4 + t^8$ so does not satisfy the Hard Lefschetz Theorem.

\end{remark}

\quash{
\begin{example}
The base change 
of \eqref{eq: homeo gr} to the base point $(0,,...,0)\in\bbR^m$
gives rise to a 
$K_c$-equivariant stratified homeomorphism 
\[T^0_\bbR\is X[\bP^1\setminus\{0\}]_{\infty\to e}\]
where $G_\bbR[\bP^1(\bbR)\setminus\{0\}]_{\infty\to e}$
is the real ind-variety of maps $\gamma:\bP^1\setminus\{0\}\to G$
such that $\gamma(\bP^1(\bbR)\setminus\{0\})\subset G_\bbR$ and $\gamma(\infty)=e$,
and $X[\bP^1\setminus\{0\}]_{\infty\to e}$ is complex ind-variety of 
maps $\gamma:\bP^1\setminus\{0\}\to X$ such that $\gamma(\infty)=e$.

Then the real arc group $G_\bbR[\bP^1(\bbR)\setminus\{0\}]_{\infty\to e}$
is the open cell $T_\bbR^0\subset\Gr_\bbR$ the intersection 
$T_\bbR^0\cap \overline S^\lambda_\bbR$

\end{example}

}


\subsection{Kostant-Sekiguchi homeomorphism for $\on{GL}_n$}
Here we explain how the prior construction recover 
Theorem 1.1 and Theorem 1.2 in  \cite{CN2} but by a completely different argument.\footnote{The argument in \cite{CN2} used quiver varieties and hyperk\"ahler rotations.}

Consider the case when $G=\on{GL}_n(\bbC)$ with real form $G_\bbR=\on{GL}_n(\bbR)$.
We  have  
$K=\on{O}_n(\bbC)$
and $K_c=\on{O}_n(\bbR)$ the 
complex and real orthogonal groups respectively.
The conjugation $\eta$ and involution $\theta$ on $\on{GL}_n(\bbC)$
are given by $\eta(M)=\overline M$ and $\theta(M)=(M^t)^{-1}$.
Let $\fg_n=\frak{gl}_n(\bbC)$ be the Lie algebra of $G=\on{GL}_n(\bbC)$, $\ft_n \subset \fg_n$  the subspace of diagonal matrices, and $\fc_n=\ft_n\gitquot\on{S}_n$  the quotient
of $\ft_n$ by the symmetric group $\on{S}_n$ on $n$ letters.
Let $\chi:\fg_n\to \fg_n\gitquot\on{GL}_n(\bbC)\is\fc_n$  be the Chevalley map.

As observed by G.~Lusztig and B.C.~Ngo,
for any $(M,z_1,...,z_n)\in\fg_n\times_{\fc_n}\ft_n(\bbR)$,
 the formula 
$\gamma:=\on{id}-z^{-1}M$ defines a map from 
$\gamma:\bP^1\setminus\{z_1,..,z_n\}\to\GL_n(\bbC)$
satisfying $\gamma(\infty)=\on{id}$.
Indeed, the collection $\{z_1,..,z_n\}$ is the set of eigenvalues of $M$
and hence the matrix $\on{id}-z^{-1}M$ is invertible for $z\in\bP^1\setminus\{z_1,..,z_n\}$.
Thus we have  an embedding
\[\fg_n\times_{\fc_n}\ft_n(\bbR)\lra\mG\is\Gr^{(n),0}\times_{\bP^1}\mathbb R^n\]
given by
\[(M,z_1,...,z_n)\to (1-z^{-1}M,z_1,...,z_n)\]
compatible with the natural projection maps
to $\ft_n(\bbR)\is\mathbb R^n$.
Moreover, by  \cite[Lemma 2.3.1]{Ngo}, the image of $\fg_n\times_{\fc_n}\ft_n(\bbR)$ in $\calG$
is a union of strata  and the restriction of those strata
to the fibers of the projection $\fg_n\times_{\fc_n}\ft_n(\bbR)\to\ft_n(\bbR)$
are unions of orbits under the natural adjoint action
of $\GL_n(\bbC)$
on $\fg_n\times_{\fc_n}\ft_n(\bbR)$.
Thus the family of involutions in Theorem \ref{family of involutions}
restricts to a 
family of involutions
\beq\label{involutions of g_n}
\alpha_a:\fg_n\times_{\fc_n}\ft_n(\bbR)\to \fg_n\times_{\fc_n}\ft_n(\bbR),\ \ \ a\in[0,1]
\eeq
satisfying the following properties:

\begin{thm}\label{KS involution}
The family of involutions 
$\alpha_a:\fg_n\times_{\fc_n}\ft_n(\bbR)\to \fg_n\times_{\fc_n}\ft_n(\bbR)$, $a\in[0,1]$ satisfy the following:
\begin{enumerate}
\item We have $\on{pr}\circ\alpha_a=\on{pr}:\fg_n\times_{\fc_n}\ft_n(\bbR)\to\ft_n(\bbR)$, for all $a\in [0,1]$.
\item $\alpha_a$ is $\mathrm O_n(\bbR)$-equivariant and takes a $\GL_n(\bC)$-orbit 
real analytically to a $\GL_n(\bC)$-orbit.
\item At $a=0$, we have $\alpha_0(M,z_1,...z_n)=(\overline M,z_1,...,z_n)$.
\item At $a=\infty$, we have $\alpha_\infty(M,z_1,...z_n)=(M^t,z_1,...z_n)$.
\end{enumerate}
\end{thm}
\begin{proof}
Only part (3) and (4) require proof, 
and they follow from the following identities:
for $\gamma(z)=\on{id}-z^{-1}M$ we have 
$\alpha_0(\gamma(z))=\overline{(\gamma(\bar z))}=
\overline{(\on{id}-{\overline z}^{-1}M)}=\on{id}-z^{-1}\overline M$
and $\alpha_1(\gamma(z))=\theta(\gamma(z))^{-1}=
(\on{id}-z^{-1} M)^{t}=
\on{id}-z^{-1}M^t$.
\end{proof}

Let $\fg_n(\bbR)$ be the space of real $n\times n$-matrices 
and let $\fg_{n}^{\on{sym}}$ be the space of $n\times n$ complex symmetric matrices.
 Theorem \ref{G=X} implies the following result which can be viewed as a 
lift of the well-known Kostant-Sekiguchi bijection
between real and symmetric nilpotent orbits to a stratified homeomorphism:

\begin{thm}\label{KS homeo}
There is an $\on{O}_n(\bbR)$-equivariant  homeomorphism 
\begin{equation}\label{eq: homeo gr}
\xymatrix{
\fg_n(\bbR)\times_{\fc_n}\ft_n(\bbR)\ar[r]^-\sim& \fg_n^{\on{sym}}\times_{\fc_n}\ft_n(\bbR)
}
\end{equation}
fitting into the diagram 
\[\xymatrix{\fg_n(\bbR)\times_{\fc_n}\ft_n(\bbR)\ar[r]^-\sim\ar[d]&\fg_n^{\on{sym}}\times_{\fc_n}\ft_n(\bbR)\ar[d]
\\\ft_n(\bbR)\ar[r]^{\on{id}}&\ft_n(\bbR)}\]
that restricts to real analytic isomorphisms 
between $\on{GL}_n(\bbR)$ and $\on{O}_n(\bbC)$
adjoint orbits on $\fg_n(\bbR)\times_{\fc_n}\ft_n(\bbR)$
and $\fg_n^{\on{sym}}\times_{\fc_n}\ft_n(\bbR)$.

\end{thm}

\begin{remark}
It follows from the construction that 
the family of involutions 
in~\eqref{involution G} 
are in fact equivariant under the natural $\on{S}_n$-action 
on $\calG\is\Gr^{(m),0}\times_{(\bP^1)^m}\bbR^m$.
Thus the involutions 
in~\eqref{involutions of g_n} are also equivariant under the natural $\on{S}_n$-action 
on $\fg_n\times_{\fc_n}\ft_n(\bbR)$. Hence Theorem \ref{KS involution} and Theorem \ref{KS homeo} have a direct analogy 
for the quotient $\fg_n\times_{\fc_n}\ft_n(\bbR)\gitquot\on{S}_n$.

On the other hand, the quotient $\fg_n\times_{\fc_n}\ft_n(\bbR)\gitquot\on{S}_n$
is isomorphic to the  subset $\fg'_n\subset\fg_n$ consisting of 
matrices with real eigenvalues. Thus in this way we recover 
Theorem 1.1 and Theorem 1.2 in  \cite{CN2}.  but by a completely different argument.

\end{remark}


{}
\end{document}